\documentclass{article}

\usepackage{amsmath}
\usepackage{amsthm}
\usepackage{amsfonts}
\usepackage{amssymb}
\usepackage{mathrsfs}

\usepackage{mathtools}
\usepackage{graphics}
\usepackage{epsfig}
\usepackage{subfigure}
\usepackage{bbold}
\usepackage{bigints}

\usepackage{authblk}

\newtheorem{theorem}{Theorem}[section]
\newtheorem{corollary}{Corollary}[theorem]
\newtheorem{lemma}[theorem]{Lemma}
\newtheorem{remark}[theorem]{Remark}
\newtheorem{definition}[theorem]{Definition}
\newtheorem{notation}[theorem]{Notations}
\newtheorem{proposition}[theorem]{Proposition}

\newcommand{\onen}{\mathbb{1}_n}
\newcommand{\onenn}{\mathbb{1}_{n-1}}
\newcommand{\zn}{\mathbb{0}_n}
\newcommand{\znn}{\mathbb{0}_{n-1}}

\newcommand{\Inn}{\textrm{I}_{n-1}}
\def\R{\mathbb R}
\def\t{\theta}
\newcommand{\mco}{\mathcal O}
\newcommand{\sww}{\sigma_{ww}}
\newcommand{\swy}{\sigma_{wy}}
\newcommand{\syy}{\sigma_{yy}}

\newcommand{\beq}{\begin{equation}}
\newcommand{\eeq}{\end{equation}}
\newcommand{\bdima}{\begin{displaymath}}
\newcommand{\edima}{\end{displaymath}}

\newcommand{\tayvec}{\left [\begin{array}{l}
\sigma_{ww} \\
\sigma_{wy} \\
\sigma_{yy} \\
\end{array}
\right ]}

\newcommand{\dtayvec}{\left [\begin{array}{c}
w_k^2-\sigma_{ww} \\
w_k y_k-\sigma_{wy} \\
y_k^2-\sigma_{yy} \\
\end{array}
\right ]}

\begin{document}

\title{Stability of Translating States for Self-propelled Swarms with Quadratic Potential}
\date{\vspace{-5ex}}
\author{Irina Popovici}
\affil{Mathematics Department, United States Naval Academy, Annapolis, MD 21402, popovici@usna.edu}

\maketitle

\begin{abstract}
  The main result of this paper is proving the stability of translating states (flocking states) for the system of $n$-coupled self-propelled agents governed by $\ddot r_k = (1-|\dot r_k|^2)\dot r_k - \frac{1}{n}\sum_{j=1}^n(r_k-r_j)$, $r_k\in \mathbb R^2$. A flocking state is a solution where all agents move with identical velocity. Numerical explorations have shown that for a large set of initial conditions, after some drift, the particles' velocities align, and the distance between agents tends to zero.  We prove that every solution  starting near a translating state asymptotically approaches a translating state nearby, an asymptotic behavior exclusive to swarms in the plane. We quantify the rate of convergence for the directional drift, the mean field speed, and the oscillations in the direction normal to the motion.
  The latter decay at a rate of $1/ \sqrt t$, mimicking the oscillations of some systems with almost periodic coefficients and cubic nonlinearities. We give sufficient conditions for that class of systems to have an asymptotically stable origin.
\end{abstract}

\section{Introduction}

The emergence of patterns in multi-agent dynamical systems has received a lot of attention since Turing featured it in the seminal paper
 \cite{Turing}; interest was reignited by Kuramoto's model of synchronization \cite{Kuramoto} and Cucker-Smale's model of flocking \cite{CuckerSmale}.
Self organization has been observed in other natural swarms (school of fish, desert locusts, pacemaker cells  , \cite{Ballerini}, \cite{Camazine},
\cite{Buhl}, \cite{Reppert}) and in artificial systems, sometimes  as a desired effect (synchronization of power grids, lasers, \cite{Hellmann}, \cite{Wang}), sometimes as detrimental (crowd-structure interactions in \cite{Bridge}).

In this paper we study the alignment in a swarm system with $n$ identical self-propelled agents having all-to-all coupling
\beq
\label{Main}
\ddot r_k = (1-|\dot r_k|^2)\dot r_k - \frac{1}{n}\sum_{j=1}^n(r_k-r_j), \; \mbox{ where } r_k(t) \in \mathbb R^2
\eeq
\begin{notation} We denote by $\displaystyle R(t)=\frac{1}{n}\sum_{j=1}^n r_j(t)$ the position of the center of mass, and by $V(t)= \dot R=\frac{1}{n}\sum_{j=1}^n \dot r_j(t)$ the velocity of the center of mass, a.k.a the mean velocity.
\end{notation}

On occasion, it is convenient to think of the motion as taking place in the complex plane $\mathbb C$ rather than $\R^2;$  we identify of a vector $r=(x, y)$ with the complex number $\mathbb r= x+i y.$
The system (\ref{Main}) becomes

\beq
\label{MainC}
\ddot{\mathbb r}_k = (1-|\dot{\mathbb r}_k |^2)\dot{\mathbb r}_k - (\mathbb r_k- R).
\eeq

Note that given constants $\mathbb c_0 \in \mathbb C$ and $\theta _0 \in \R$, if $\displaystyle \left ( \mathbb r_k (t) \right )_{k=1..n}$ is a solution for (\ref{MainC}), then the vectors
$\displaystyle e^{i \theta _0} \mathbb r_k (t) + \mathbb c_0$ and the
complex conjugate vectors $\overline{\mathbb r_k (t)} $ satisfy (\ref{MainC}) as well. Thus the systems (\ref{Main}) and  (\ref{MainC}) are translation, rotation and reflection invariant.

It has been shown that the system's coupling strength is enough to ensure that eventually all agents move in finite-diameter configurations (\cite{M-Popovici}), yet not too strong to lock the dynamics into a unique pattern of motion. We note that the self-propelling term $(1-|\dot r_k|^2)\dot r_k$ brings energy into the system when a particle's speed is below one, and dissipates energy when the particle moves at high speed. This suggests that if agents are to align or synchronize, they would select patterns where all agents move at roughly unit speed. Numerical experiments show that for a large set of initial conditions the system reaches a limit state that has the center of mass either moving uniformly (at unit speed), or becoming stationary, the jargon for these configurations being flocking/translating and milling/rotating. (There are no generally accepted definitions for these terms; see definitions  \ref{def_rotating},
\ref{def_flocking} for how they are used in this paper). We want to emphasise that unlike the classical Cucker-Smale flocks \cite{CuckerSmale} or Kuramoto oscillators \cite{Kuramoto}, the center of mass for (\ref{Main}) does {\em not} have uniform motion:
\beq
\label{MainAcc}
\dot V= \frac{1}{n}\sum_{j=1}^n(1-|\dot r_j|^2)\dot r_j.
\eeq

Theoretical results in \cite{M-Popovici} and in this paper (Figures \ref{Drift}, \ref{Drift2} and Remark \ref{driftremark}) show that as the center of mass moves toward its limit state  it has a not-insignificant drift in its location or in its direction of motion. Let $\Theta(t)$ denote the polar angle of the mean velocity, so $\displaystyle V(t)=|V|e^{i\Theta(t)}.$

\begin{figure}[h!]
\centering
\includegraphics[width=.4\textwidth]{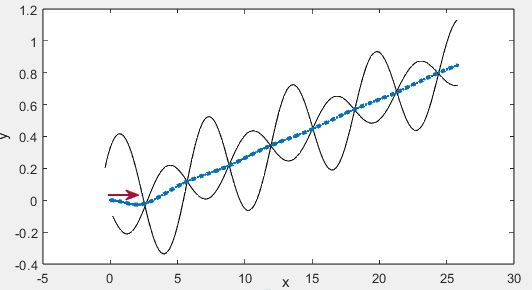}
\caption{ The position of two of the particles of a multi-agent swarm  (\ref{Main}), shown as the solid black curves. Initially the center of mass  has horizontal velocity and unit speed, as shown by the red vector. The flock's eventual velocity $V_{\infty}$ has drifted away from the horizontal, as illustrated by the location of the center of mass (the blue square markers).  The speed $|V|$ approaches one at a rate of $1/t.$}
\label{Drift}
\end{figure}

\begin{figure}[h!]
\centering
\includegraphics[width=.6\textwidth]{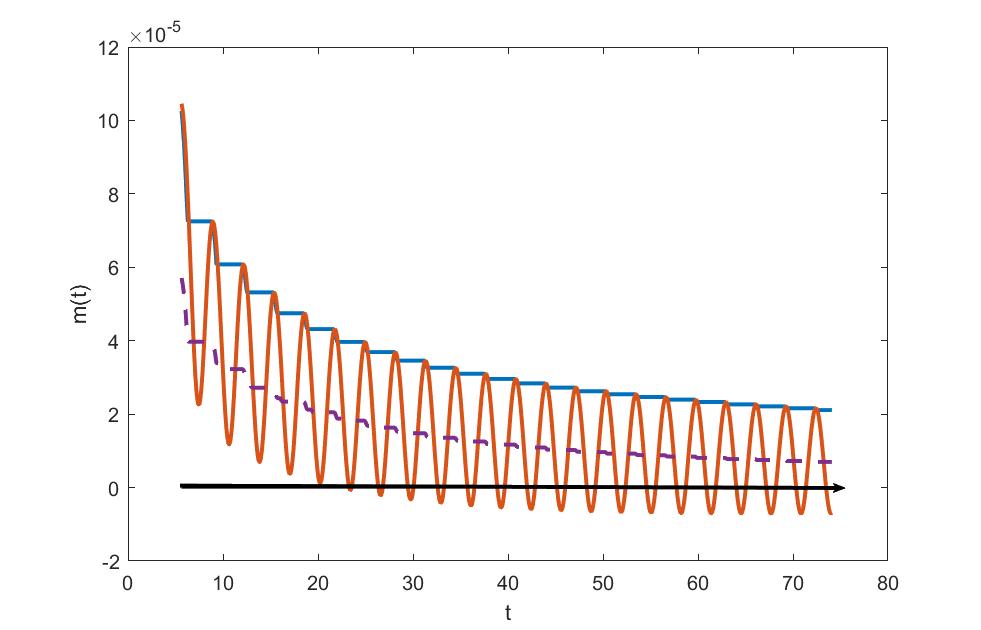}
\caption{ The change in the direction of motion for a swarm with four agents. The oscillatory curve is the graph of $m=\frac{d\Theta}{dt}$ where $\Theta $ is the polar angle of the mean velocity $V.$
 The dotted line shows the mean value of $m$ over a time window of size $2\pi.$
The flock's eventual velocity $V_{\infty}$ has a convergent polar angle, with an overall counterclockwise drift equal to $\Theta_{\infty}-\Theta_0= \int_0^\infty m dt>0.$}
\label{Drift2}
\end{figure}

The center of mass is very sensitive to perturbations, in the sense that small changes in initial conditions  could result in large scale drift of the center of mass just on account of changing the direction of motion. In order to control the evolution of small perturbations it may be more expedient to focus on the velocities  $V$ and $v_k$ and to recover the information about $R$ and $r_k$  via integration.
The velocity-acceleration system for  $v_k=\dot r_k$, obtained by differentiating (\ref{Main}) is:
\beq
\label{VelocityAcc}
\ddot v_k = (1-|v_k|^2)\dot v_k - 2(v_k^T \dot v_k) v_k- (v_k-V)
\eeq
where the superscript $T$ denotes transposition.

\subsection{Special Configurations}

Second order models of self-propelled  agents with all-to-all gradient coupling  have been extensively studied, notably  by researchers in Prof. Bertozzi's group (\cite{Topaz}, for continuous models, and \cite{BertozziCaltech} for agent-based models). Let $K: [0, \infty) \to  \R$ be a continuous scalar function, differentiable on $(0, \infty).$ Define $U:\R^d\to \R$ as $U(x)=K(|x|).$ With the possible exception of the origin, the potential $U$ is differentiable with $\nabla U(x)= K'(|x|) \frac{x}{|x|}. $ Consider the system with identical self propelling terms and rotationally symmetric potential:
\beq
\label{gradient}
\ddot r_k=(\alpha -\beta |\dot r_k|^2)\dot{r}_k - \frac{1}{n}\sum_{j\neq k }\nabla U(x_k-x_j)
\eeq
where $\alpha, \beta$ are positive constants.

The system (\ref{Main}) is in the class of (\ref{gradient}) with $K(r)=\frac{1}{2}r^2$ and $\alpha=\beta =1.$ Other well known classes of radial potentials are Morse potentials where $K$ is the difference between two exponential decays (see \cite{Bertozzi2006}, \cite{vonBrecht}), and power law potentials where $K(r)= c_ar^a-c_br^b, $ ( \cite{Bertozzi2015}, Section 8; \cite{Carrillo};  \cite{M-Popovici}, 3.3-3.5). Given the symmetry of (\ref{gradient}) it is natural to expect that all emergent patterns feature rotational symmetries; in that spirit, with the exception of \cite{M-Popovici} and \cite{K-Popovici}, most theoretical results address particles that are uniformly distributed around their center of mass and their similarly {\em restricted} perturbations, or have assumptions about the degeneracies of the system that for (\ref{Main})
only hold for $n \leq 2$, \cite{C-S-Martin}. We emphasize that we make no uniform distribution assumption in this paper, which explains the sharp distinction between the rates of convergence to limit configuration: exponential rates  obtained by authors working exclusively in the subset with rotational symmetries, versus $\frac{1}{\sqrt t}$ in this paper (and in \cite{K-Popovici}). In fact (\ref{Main}) has very few configurations with symmetries:

\begin{proposition}
\label{NoSymmetries}
If among the $n$ agents of the system (\ref{Main}) there exist $p \geq 3 $  agents (assumed to be the first $p$) whose motion is symmetrically distributed about the center $R$ of the swarm, i.e.
\bdima
\mathbb r_k(t)= R(t)+e^{2\pi i \frac{(k-1)}{p}}(\mathbb r_1(t)-R(t)) \; \mbox{ for }\; k=1 ..p,
\edima
then either $\mathbb r_1(t)=\dots \mathbb r_p(t)$ for all $t\in \R$, or $R(t)=R(0)$ for all $t\in \R.$
\end{proposition}
\begin{proof} To improve the flow of the paper, the proof is presented in the Appendix.
\end{proof}

It was shown in \cite{K-Popovici} that the only possible configurations of (\ref{Main})  with a stationary center of mass come from partitioning the agents $\{ 1, \dots n\}$ into subsets $\mathcal F, \mathcal L, \mathcal R$ where the agents in $\mathcal F$ have fixed positions coinciding with the center $R(0)$ at all times; each agent in $\mathcal L $ oscillates on a segment centered at $R(0)$,  satisfying the scalar  Lienard equation  $\ddot a= ( 1- \dot a ^2) \dot a -a$ (which has an attracting cycle); each agent in $\mathcal R$  is asymptotic to a rotation $e^{\pm it} e^{i\theta_{\infty, k}}.$

\begin{definition}
\label{def_rotating}
A solution $\displaystyle \left (  r_k (t) \right )_{k=1..n}$  for (\ref{Main}) is called a rotating state if there exist angles $\theta_{0,k}$  with $\displaystyle \sum_{k=1}^n e^{i\theta_{0,k}}=0$ and a constant $ c \in \mathbb C$ such that
$ r_k (t) = e^{i\theta_{0,k}}e^{it} + c $ for all $t$ and all $k$ (a counterclockwise rotating state), or
$r_k (t) = e^{i\theta_{0,k}}e^{-it} + c $ for all $t$ and all $k$ (a clockwise rotating state).
\end{definition}
Some authors, but not all, use the term mill for a rotating state; most authors define a mill as a rotating state whose angles $\theta_{0,k}$ are uniformly distributed in $[0, 2\pi]$, a.k.a having rotational symmetry. {\em All} the rotating states of (\ref{Main}) are stable, per \cite{K-Popovici}. Configurations with {\em rotational symmetries} for this system and for other radial potentials have been extensively studied: see \cite{Bertozzi2007}, \cite{Bertozzi2006}, \cite{vonBrecht}, \cite{Bertozzi2015} 
for some early theoretical and experimental results.

\begin{definition}
\label{def_flocking}
A solution $\displaystyle \left ( r_k (t) \right )_{k=1..n}$  for (\ref{Main}) is called a flocking state if $ \dot{ r}_{k_1} (t)= \dot{r}_{k_2}(t)$ for all $k_1, k_2 $ in $1..n $. (Necessarily the common velocity is that of the center of mass.)
\end{definition}
Note that for a flocking state $ \ddot{ r}_k = \dot{ V} .$ If we subtract the
equation for $ \ddot{ r}_{k_1}$ from that of $ \ddot{ r}_{k_2}$ we get that $ r_{k_1}= r_{k_2}$ for all $k_1, k_2 $ meaning that all particles have collapsed onto the center of mass, moving together according to
 $$\dot{ V}= (1-| V|^2) V. $$

 Unless $ V(t)=0 $ for all $t$ (the trivial solution), the velocity of the flocking solutions maintains the initial direction of $ V(0)= | V(0)| e^{i \Theta_0} $ with speed approaching unit speed at a rate of $e^{-2t}$ (since $| V|^2$ solves the logistic equation $ \displaystyle \frac{d}{dt}( | V|^2)= 2 (1-| V|^2)| V|^2 \; ). $ A nontrivial flocking state has velocity approaching $e^{i \Theta_0} $ exponentially fast.
 Moreover, if $ c_{\infty}$ denotes the absolutely convergent integral
  $\displaystyle  c_{\infty}= \int _0^\infty \left [ V (\tau)- e^{i \Theta_0} \right ]d \tau $ then the flocking state's center of mass has nearly rectilinear motion:
  $\lim _{t \to \infty} \left [R(t)-(  c_{\infty} + t  e^{i \Theta_0} ) \right]= 0.$
  The limit holds since
   $ R(t)- t  e^{i \Theta_0} - c_{\infty} = \int _0^t \left [V (\tau)- e^{i \Theta_0} \right ] d \tau -\int _0^\infty \left [ V (\tau)- e^{i \Theta_0} \right ] d \tau = -\int _t^\infty \left [ V (\tau)- e^{i \Theta_0} \right ] d \tau .$

\begin{definition} A solution $\displaystyle \left (  r_k (t) \right )_{k=1..n}$  for (\ref{Main}) is called a translating state if there exists $\Theta_0 \in [0, 2\pi]$ such that  $ \dot{ r}_{k}(t)= e^{i \Theta_0} $ for all $k=1..n.$
 (Necessarily, for a translating state there exists $c_0 \in \mathbb C$ such that $ r_k (t)=  c_0+ te^{i \Theta_0} $ for all $k.$)
 \end{definition}

The main result of this paper is proving that the translating states are stable, in the sense that if the particles in (\ref{Main}) start close to each other, with velocities near some  unit vector $(\cos \Theta _0, \sin \Theta _0)$ then the particles remain close to each other, with velocities that are asymptotic to a common unit vector, whose direction does not drift far from $\Theta _0, $ as stated in Theorem \ref{MainTh}. We also quantify the rate of convergence to the limiting configuration (Corollary \ref{RateCor}), and we show that it is at a rate of $\displaystyle \frac{1}{\sqrt t}, $ much slower than the conjectured exponential rate.

\begin{theorem}
\label{MainTh} Consider the system (\ref{Main}) with initial conditions $r_k(0)=r_{k,0} $ and $\dot r_k(0)=\dot r_{k,0}. $ For any $\epsilon >0$ there exists $\delta >0$ such that if $r_{k,0} $ and $\dot r_{k,0}$ satisfy:
$ |r_{k_1,0}- r_{k_2,0}| <\delta $ for all $k_1, k_2$ in $1..n$ ; $1-\delta<|V_0|< 1+\delta$; $ |\dot r_{k,0}-V_0|  <\delta $ for all $k$ in $1..n$, then  for $t>0$ :
\beq
\label{FirstMain}
\begin{array}{l}
|r_{k_1}(t)- r_{k_2}(t)| <\epsilon \; \mbox{ for all } k_1, k_2 \; \mbox{ in } 1..n \\
1-\epsilon<|V(t)|< 1+\epsilon \\
 |\dot r_k(t)-V_0|  <\epsilon \; \mbox{  for all } k \mbox{ in } 1..n
 \end{array}
 \eeq
Moreover,
\beq
\label{MainZeroDr}
\lim _{t \to \infty} |r_{k_1}(t)- r_{k_2}(t)| =0 \; \mbox{ for all } k_1, k_2 \mbox{ in } 1..n
\eeq
 and there exists $\Theta_{\infty}$ in $[0, 2\pi]$ with $|V(0)-(\cos \Theta_{\infty}, \sin \Theta_{\infty})|<\epsilon$ such that
 \beq
 \label{MainZeroV}
 \lim _{t \to \infty} \dot r_k(t)=  \lim _{t \to \infty} V(t)= (\cos \Theta_{\infty}, \sin \Theta_{\infty}).
 \eeq
\end{theorem}

The proof of Theorem \ref{MainTh} is completed in three parts, corresponding to three dimension reductions of the system: form $4n$ to $4n-3$ in   Section \ref{SectionMovingFrame}, to $2n$
in  Section \ref{Section_CentralManifold}, and to $n$ in Section \ref{Section_Stability}, with the stability of the translating states for (\ref{Main}) being equivalent to that of the origin for the reduced systems.

The dynamical system  produced in Section \ref{SectionMovingFrame} evolves in a $4n-3$ dimensional invariant subspace for a system taking the form
 $\dot w = Aw + G(w)$  where $A$ is the Jacobian matrix and $G$ is a nonlinear function whose Maclaurin  expansion has terms or order 3 or higher. Counted with multiplicity, the eigenvalues of $A$ consist of $n$ pairs of purely imaginary roots $\pm i$, and $2n-1$ eigenvalues with negative real parts.
The stability of the origin for systems  whose Jacobian matrices have  $q $ pairs  of purely imaginary roots $(\pm i \omega_1, \dots , \pm i \omega _q)$ and $d-2q$ roots with negative real part has been extensively investigated; the problem is commonly known as the critical case with purely imaginary roots. Most authors apply elaborate iterative transformations from normal form theory to obtain systems for which stability criteria such as Molchanov and Lyapunov's second method can be employed. These transformations require that the frequencies $\omega_k$ satisfy non-resonance conditions which are violated in our setting,(\cite{Zuyev},
\cite{Krasilnikov}), or alternatively, they assume the a priory existence of some quadratic
Lyapunov functions ( \cite{Grushkovskaya}), rendering their methods inapplicable here despite concluding similar decay rate of $1/\sqrt t$ towards the origin. Our approach 
draws inspiration from a class of time-dependent systems with cubic nonlinearities instead:

\beq
\dot x_k = b_k(t)x_k^3  + c_k(t) x_k  \frac{1}{n} \sum \limits_{l=1}^n d_l(t) x_l^2+ R_{k}(x, t), \; \mbox{ for }  k=1 \dots n,
\label{AlmostP1}
\eeq
where  $x=(x_1, \dots x_n)^T$ is vector-valued from $\R$ to $\R^n$, the coefficient functions $ b_k, c_k , d_k: \R \to \R$ are continuous and the remainder functions $R_k$ defined on $\R^n\times \R $ satisfy the $\textsl{o}_3$ condition: there exists a function $\epsilon :[0, \infty) \to \R$ with $\displaystyle \lim _{r->0} \frac{\epsilon (r)}{r^3}=0$ such that  $|R_k(x, t)| \leq \epsilon (|x|) $  for all $(x, t)\in \R^n\times \R .$
In Section \ref{Subsection_TimeDependent} we provide easily verifiable conditions on the coefficient functions that sufficiently ensure the asymptotic stability of the solution $x=\zn$, with convergence rate of order $\frac{1}{\sqrt t}.$ We adapt the arguments from Section \ref{Subsection_TimeDependent} to complete the proof of Theorem \ref{MainTh}, in Section \ref{SubSection_Stability}. We use the following to facilitate notations:

\begin{notation} (i)  Given a positive integer $d$ denote the column vectors with all-$1$ entries and all-$0$ entries in $\R^d$ by
\beq
\label{FundamVec}
\mathbb{1}_d= (1, 1, \dots 1)^T, \; \mbox{ and } \mathbb{0}_d= (0, 0, \dots 0)^T.
\eeq
 Given positive integers $ d_1, d_2$ denote the $d_1 \times d_2$ matrices of all ones and all zeros by $\displaystyle \mathbb{1}_{d_1, d_2}= \mathbb{1}_{d_1}\mathbb{1}_{d_2}^T, $ and
  $\displaystyle \mathbb{O}_{d_1, d_2}= \mathbb{0}_{d_1}\mathbb{0}_{d_2}^T. $ We may use just $\mathbb{O}$ for a matrix if its dimensions are clear form context.
  Denote the $d \times d$ identity matrix by $I_d.$ Denote by $H_d$ the hyperplane orthogonal to $\mathbb{1}_d $ in $\R^d.$

  (ii) Given (column) vectors $ a\in \R^{d_1}$ and $ b\in \R^{d_2}$ denote by $\mbox{col}(a, b) \in \R^{d_1+d_2}$ the column vector $(a_1, \dots a_{d_1}, b_1, \dots b_{d_2})^T.$
\end{notation}

We conclude this introductory subsection by examining how other authors interpret
stability for limit patterns and comparing their settings to our
results, as stated in Theorem \ref{MainTh}, Remarks \ref{remarkVinfty},  and
\ref{driftremark}, since there is no universally accepted definition.

For Cucker-Smale (CS) models with communication function $\psi$
\bdima
 \ddot r_k= \frac{1}{n}\sum_{j=1}^n \psi(|r_j-r_k|) (\dot r_j-\dot r_k)
 \edima
the generally accepted definition of {\it asymptotic flocking} requires two conditions:
\begin{itemize}
\item Velocity alignment: $\lim _{t \to \infty} \sum \limits_{j=1}^n |\dot r_j(t) -V(t)|^2 =0$
\item Forming a finite-diameter group: $\sup _{0\leq t < \infty} \sum \limits_{j=1}^n | r_j(t) -R(t)|^2 < \infty.$
\end{itemize}
Cucker-Smale models have constant mean velocity $V(t)=V(0)$ so the existence of a limiting value of $V$ is automatic; the problem reduces to an equivalent model with $V(t)=0.$ For the classical communications functions
$\psi (s) =\frac{\alpha}{s^\beta}$ and $\psi (s) =\frac{\alpha}{(1+s^2)^{\beta/2}}$ with $0 \leq \beta \leq 1$  it is proven in \cite{Liu} and \cite{CuckerSmale} that there is {\it unconditional flocking} in the sense that all solutions exhibit asymptotic flocking with $ |\dot r_j(t) -V(t)|$ being exponentially small.

For Kuramoto oscillators (KO)
\bdima
\dot \theta _k = \omega _k + \frac{K}{n}\sum_{j=1}^n \sin(\theta_j-\theta_k|) , \; \; \theta _k(t) \in \R
\edima
the mean frequency $\omega= \frac{1}{n}\sum_{j=1}^n \omega _j$ is constant, ensuring that the mean phase $\frac{1}{n}\sum_{j=1}^n \theta_j $ is linear in time.
For (KO) synchronization plays the role of (CS)'s velocity alignment,  requiring that  the angular frequencies $\dot \theta _k$  converge to the mean frequency.
It was shown that under some mild conditions on the initial distribution of phase angles, even non-identical oscillators synchronize \cite{Chopra}, and that the exponential rate of convergence proven in \cite{Kim} for identical oscillators (i.e. when all $\omega _k=\omega $ ) extends to more general systems. Not all systems synchronize, and some do at much slower rates. Moreover, \cite{Kim} shows that identical oscillators with initial phase angles confined in an semicircle have {\em phase} angles that  converge exponentially to a common $\theta_\infty.$

For a swarm of identical planar agents coupled with a power law potential as in  (\ref{gradient}), some authors  (\cite{Carrillo2}) define flocking configurations as being solutions where all agents have rectilinear motion, with a common velocity: $r_k(t)=r_{k,0} +v_0 t$ for all $k.$  Necessarily $ v_0= \sqrt{\frac{\alpha}{\beta}} e^{i \theta _0}$ for some $ \theta _0 \in [0, 2\pi).$
Let $r_{\star}=(r_{k,0}) _{k=1..n}$ in $\mathbb R ^{2n}$ denote the initial locations of a flocking configuration. In \cite{Carrillo2} the authors  define the stability of the flocking state    evolving from $r_{\star} $  by referring to a manifold of configurations. They define the {\it flock manifold} $ \mathcal F(r_{\star}) \subset \R^{4n}$ as
\bdima
\mathcal F(r_{\star})= \{ (e^{i \phi}
 r_{k,0}+ t\sqrt{\frac{\alpha}{\beta}} e^{i \theta}, \sqrt{\frac{\alpha}{\beta}} e^{i \theta}) _{k=1..n} , \mbox{ for all } \phi \in  [0, 2\pi), \theta \in [0, 2\pi) , t>0\}
\edima
and the 'local asymptotic stability' of $ \mathcal F(r_{\star}) $ as the existence of a $\delta-$ neighborhood $U$ of $ \mathcal F(r_{\star}) $ such that for initial conditions $r_0, \dot r _0$ in $U,$ the distance from the state variables $r(t), \dot r(t)$ to the manifold $ \mathcal F(r_{\star}) $ approaches zero as $t\to \infty.$ Note that this is a rather weak condition on the behavior of $V(t)$, as it only requires that $|V|\to \sqrt{\frac{\alpha}{\beta}}$ without imposing any conditions on $V$  having a steady direction.

\subsection{Obstacles on the way to traditional analysis}

One of the difficulties in working with translating and flocking states is the fact that they are neither fixed points nor periodic cycles. Moreover, for (\ref{Main}) the perturbations of the translating states have centers of mass that do not move with constant velocity (or in a predetermined direction) as in Cucker-Smale where excising a linear-in-$t$ term out of the motion transforms the flocking states into fixed  points.

In \cite{M-Popovici}  the authors prove that there exit constants $C_r, C_v, C_a$ that depend only on the number of particles, such that every solution to (\ref{Main} ) satisfies $|r_k(t)-R(t)|\leq C_r, \; |\dot r_k(t)-V(t)|\leq C_v,$ and $|\ddot r_k(t)-\dot V(t)|\leq C_a$ for large enough $t.$ This suggests two possible workarounds for the unboundedness of flocking states: study (\ref{VelocityAcc}) as a substitute for (\ref{Main}) or change the unknown functions of (\ref{Main} ) from $r_k$ and $\dot r_k$ to $r_k(t)-R(t)$ and $v_k=\dot r_k.$ Both approaches lead to dynamics evolving in a compact phase space, with the translating states of (\ref{Main} ) corresponding to fixed points with all velocities $v_k$ equal to the mean speed $V_0, $ a {\em unit} vector: $|V_0|=1.$

This subsection briefly addresses some remaining obstacles in pursuing these model modifications. Although this subsection helps set the tone for the dimension reduction and the proof of Theorem \ref{MainTh} that begins in  Section 2, the proof of the main results are independent of it.

\begin{remark} Any solution $\displaystyle \left ( r_k (t) \right )_{k=1..n}$  for (\ref{Main}) produces velocity vectors $v_k =\dot r_k$ that solve (\ref{VelocityAcc}). The converse is not true.
\end{remark}

\begin{proof} (\ref{VelocityAcc}) is obtained from differentiating (\ref{Main}). We provide a counterexample for the converse: let $V_0$  be any nonzero vector in the plane with $|V_0|\neq 1.$  Let $v_k (t) =V_0$ for all $t$ and all $ k=1..n.$  The vectors $v_k$ are a solution to (\ref{VelocityAcc}) that can't be realized as derivatives of solutions to (\ref{Main}), because if they did, they would be velocities for a flocking state whose asymptotic speed is neither zero nor one, an impossibility.
\end{proof}

\begin{proposition}
\label{Unstable}
Every fixed point of (\ref{VelocityAcc}) with $|V_0|=1$ is unstable.
\end{proposition}
Note that the fixed points of (\ref{VelocityAcc}) with $|V_0|=1$  correspond to the translating solutions of (\ref{Main}).
We want to emphasise that the perturbations that destabilize (\ref{VelocityAcc}) don't correspond to derivatives of solutions to (\ref{Main}). The main theorem (\ref{MainTh}) in this paper makes this point.

\begin{proof} First note that a point $\displaystyle \left ( v_{k,0},\dot{ v}_{k,0} \right  )_{k=1..n}$  is a fixed point for
(\ref{VelocityAcc}) if and only if
 $ \dot{v}_{k,0}=0$ for all $k$ and  $  v_{k,0}= V_0 $ for all $k $, where  $V_0$ is the initial mean  velocity.  (Necessarily for such a fixed point  $\dot{v}_{k,0}=0$ for all $k$;
we get from (\ref{VelocityAcc}) that $0=v_{k,0}-V_0$ for all $k$ so all  the velocities are equal.)

In order to prove that the fixed points of the system (\ref{VelocityAcc}) with $|V_0|=1$ are unstable it is enough to do so when $V_0$ is the unit vector in the $x-$ axis direction (due to the rotation invariance). Note that the linear subspace
\newline $\{ \mbox{col}(c_1 \onen, c_2 \onen, c_3 \onen, c_4 \onen), \; c_1, c_2, c_3, c_4 \in \R\}$ of  $\R^{4n}$ is invariant under the flow of  (\ref{VelocityAcc}), allowing us work within the class of solutions
that satisfy $v_{k}(t)=V(t)$
for all $k=1..n$ and all $t$.
Denote by $v_x, v_y, a_{x}$ and $a_{y}$ the horizontal and vertical components of $V$  and  $\dot V$, respectively. In this class (\ref{VelocityAcc}) becomes:
\beq
\label{EqualAV}
\begin{array}{ll}
\dot{v_x}=& a_x \\
\dot{v_y}=& a_y\\
\dot{a_x}=&(1-3v_x^2-v_y^2)a_x-2v_x v_y a_y \\
\dot{a_y}=&(1-v_x^2-3v_y^2)a_y-2v_x v_y a_x .
\end{array}
\eeq
The system (\ref{EqualAV}) has two conserved quantities:
\bdima
a_x-(1-v_x^2-v_y)^2 v_x= c_{0x}  \; \mbox{ and } a_y-(1-v_x^2-v_y)^2 v_y= c_{0y}
\edima
where the constants of motion $c_{0x}=a_x(0)-(1-v_x(0)^2-v_y(0))^2 v_x(0) $ and $c_{0y}=a_y(0)-(1-v_x(0)^2-v_y(0))^2 v_y(0)$ are small if $v_x(0)$ is near one and $v_y(0)$ is near zero.
Fix the perturbation parameter  $c_{0y}=0$ and allow $c_0=c_{0x}$ to vary near zero; this reduces the study of (\ref{EqualAV}) to that of the 3-dimensional system
\beq
\label{EqualV}
\begin{array}{ll}
\dot{v_x}=& (1-v_x^2-v_y^2)v_x + c_{0}\\
\dot{v_y}=& (1-v_x^2-v_y^2) v_y\\
\dot{c_0}=& 0.
\end{array}
\eeq
For $c_{0}=0$ all the points from the unit circle in the $(v_x,v_y)$ plane are fixed points.
 For each $c_0 \neq 0$ , small, the system (\ref{EqualV}) has three fixed points, one fixed point $(v_\star(c_0), 0, c_0)$ near $M(1,0, 0)$ and two other fixed points near $(0,0, 0)$ and $(-1,0, 0).$
$ v_\star(c_0)$ satisfies $v_\star(1-v_\star^2)=-c_0.$

The system (\ref{EqualV}) can be seen as a one-parameter family of planar systems. Its stability and bifurcations are well known: a summary of the ensuing dynamics can be found in [Gu], page 70 (their parameter $\gamma $ is $-c_0$). For $c_{0}<0 $ small, the fixed point $(v_\star , 0)$ has a positive eigenvalue ($\lambda= 1-v_\star ^2 $), and the flow of (\ref{EqualV}) has heteroclinic cycles with a saddle connection near $(1,0)$, a stable node near $(-1,0)$ and an unstable node near $(0,0)$).

The locations of the aforementioned fixed points are shown Figure \ref{FigPertVA}, with $c_0$ on the vertical axis: the curve $(v_\star(c_0), 0, c_0)$ through $M(1,0, 0)$ is illustrated in red with two $\star$ markers; also in red are the almost-vertical curves of fixed points near the fixed points $(0,0)$ and $(-1,0)$ in the $c_0=0$ plane. The fixed points on the unit circle in the plane $c_0=0$ use $\times$ markers. All the other curves in Figure \ref{FigPertVA} are trajectories within three representative planes ($c_0<0, c_0=0, $ and $c_0>0$).

\begin{figure}[h!]
\centering
\includegraphics[width=.45\textwidth]{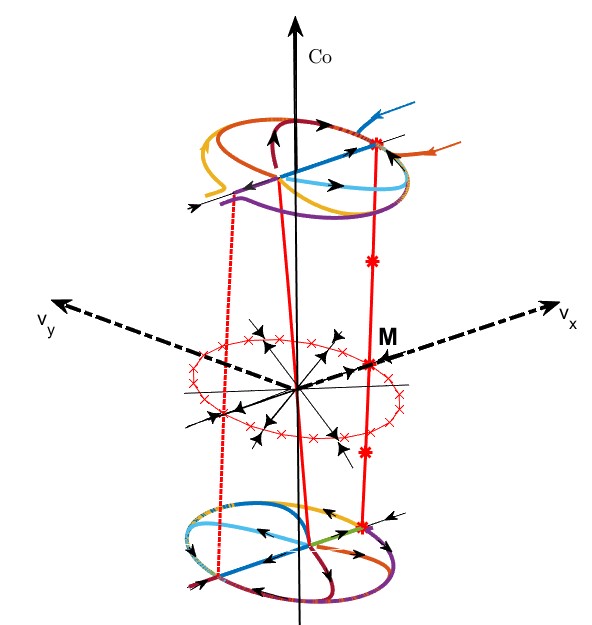}
\caption{
The flow of (\ref{EqualV}) for three  values of $c_0.$
The system (\ref{VelocityAcc}) subsumes (\ref{EqualV}). The dynamics near the fixed point of (\ref{VelocityAcc}) with $V_0=(1,0)$ corresponds to the dynamics of (\ref{EqualV}) near $M(1, 0, 0),$ which is unstable on account of having saddles and heteroclinic cycles near $M$ in the region $c_0<0.$ }
\label{FigPertVA}
\end{figure}

The dynamics of (\ref{EqualV}) in the region $c_0<0$ near $M(1,0,0)$ makes the fixed point $(1,0,0,0)$ of (\ref{EqualAV}) unstable; thus the fixed point $(\mathbb{1}_n, \mathbb{0}_n, \mathbb{0}_n,\mathbb{0}_n)$ of   (\ref{VelocityAcc}) is unstable.

\end{proof}

\vskip3mm The second natural approach for transforming the unbounded solutions of (\ref{Main}) and its translating states into bounded trajectories and fixed points is to work with $\sigma_k=r_k- R=(\sigma_{x,k}, \sigma_{y,k})$ and $\dot r_k= (v_{x,k}, v_{y,k}) $ as the unknown functions. Note that since $r_n -R =-\sum _{k=1}^{n-1} \sigma_k $ we can eliminate it from our unknowns, and work within a $4n-2$ dimensional phase space (if we were to keep the redundant $r_n -R$ as a state variable, we would have added two more zero eigenvalues to the system).

The system (\ref{Main}) becomes:

\beq
\label{SigmaV}
\begin{array}{ll}
\dot{\sigma_{x,k}}=& v_{x,k}-\frac{1}{n}\sum _{j=1}^n v_{x,j} \\
\dot{\sigma_{y,k}}=& v_{y,k}-\frac{1}{n}\sum _{j=1}^n v_{y,j} \\
\hline\\
\dot{v_{x,k}}=&(1-v_{x,k}^2-v_{y,k}^2)v_{x,k} - \sigma_{x,k} \\
\dot{v_{y,k}}=&(1-v_{x,k}^2-v_{y,k}^2)v_{y,k} - \sigma_{y,k} \\
\dot{v_{x,n}}=&(1-v_{x,n}^2-v_{y,n}^2)v_{x,n} +  \sum _{k=1}^{n-1} \sigma_{x,k} \\
\dot{v_{y,n}}=&(1-v_{x,k}^2-v_{y,k}^2)v_{y,n} + \sum _{k=1}^{n-1} \sigma_{y,k}
\end{array}
\eeq
where $k=1.. (n-1).$

The translating states of (\ref{Main}) with $V(t)=(\cos \theta _0, \sin \theta _0)$  correspond to the fixed points $P_{\theta _0}$ with
$\sigma_x = \sigma_y= \znn,$ and $ v_x = \cos \theta _0\; \onen, \; v_y= \sin \theta _0 \; \onen.$

  Focus on the dynamics near $P_0=\mbox{col}(\znn, \znn, \onenn, \znn, 1, 0).$ The other fixed points are obtained from $P_0$ via rotation of angle $\theta _0.$ The linearization of (\ref{SigmaV}) at $P_0$ yields the Jacobian matrix
\bdima
J=\left [\begin{array}{ll|ccrr}
\mathbb{O}& \mathbb{O} & B_{n-1}&\mathbb{O} &-\frac{1}{n}\onenn & \znn \\
\mathbb{O}& \mathbb{O} & \mathbb{O} & B_{n-1}&\znn & -\frac{1}{n}\onenn \\
\hline\\
-\Inn & \mathbb{O}& -2 \Inn & \mathbb{O}&\znn & \znn\\
\mathbb{O}& - \Inn & \mathbb{O} & \mathbb{O}& \znn & \znn\\
\onenn ^T & \znn ^T &   \znn ^T & \znn ^T &-2 &0\\
\znn ^T & \onenn ^T &  \znn ^T & \znn ^T & 0&0
\end{array}
\right ]
\edima
where $B_{n-1}$ is the square matrix $ \Inn-\frac{1}{n}\mathbb{1}_{n-1, n-1}$ and $\mathbb{O}=\mathbb{O}_{n-1, n-1}.$
We get that
$$ \mbox{det}(J-\lambda I_{4n-2}) = \lambda( 1+ \lambda^2)^{n-1}(2+\lambda)(1+\lambda)^{2n-2}.$$

The eigenvector for $\lambda =0$ is $\mbox{col}(\znn, \znn, \znn, \onenn, 0,1)$; it has the direction of the line tangent to the curve of fixed points  $\{ P_{\theta_0}, \; \theta_0 \in [0, 2\pi]\}$
at $P_0.$ The eigenvalue $\lambda =0$ captures the rotation invariance of the model.
Note that due to the $\pm i $ eigenvalues there are more degeneracies associated with (\ref{SigmaV}) than the symmetries of the problem.

If in the system (\ref{SigmaV}) all particles' $y-$components are initially zero, they stay zero forever, with the remaining $2n-1$ dimensional system in $\sigma_x, v_x, v_{x, n-1}$ having the Jacobian $J_x$ at $\sigma_x=\znn,  v_x =\onenn, v_{x, n-1}=1$ equal to
\bdima
J_x=\left [\begin{array}{lll}
\mathbb{O}&  B_{n-1}  & -\frac{1}{n}\onenn \\
-\Inn &  -2 \Inn &  \znn\\
\onenn ^T & \znn ^T  &-2
\end{array}
\right ]
\edima
and characteristic polynomial $(\lambda +2) (\lambda +1)^{2n-2}.$ In the absence of a $y-$ component, all trajectories approach the fixed point $P_0$ exponentially fast, meaning
 that the entire stable manifold for (\ref{SigmaV}) at $  P_0$ consists of the trajectories with no $y$ component.
 The flow on the  $2n-1$ dimensional central manifold captures the {\em drift} in the direction of motion, and the perturbations/ oscillations that are {\em orthogonal} to the direction of motion.

In order to understand the dynamics, one needs a suitable approximation of a central manifold map $h$ whose graph is a central manifold, typically done using Taylor polynomials (as in \cite{Carr}, pages 3-5).  Unfortunately, Taylor approximations are ill fitted to capture the correct scales of the flow in the presence of non-isolated fixed points, which is the case for (\ref{SigmaV}). A presentation of that inadequacy is in \cite{K-Popovici} \footnote{\cite{K-Popovici} provides an alternate construction for the central manifold for systems with non-isolated fixed points having no purely imaginary eigenvalues. Their method does not apply in the context of (\ref{SigmaV}).}  The predicament of having non-isolated fixed points and imaginary eigenvalues of high multiplicity requires we further refine the model until {\it all} the translating states are identified with a {\it single} fixed point. That is done in Section \ref{SectionMovingFrame}.

\begin{remark} If we allow the motion to take place in space (i.e replacing the assumption that $r_k \in \R^2$ from (\ref{Main}) with $r_k \in \R^3$), there exist initial conditions spawning solutions that fail the asymptotic conditions (\ref{MainZeroDr}) and (\ref{MainZeroV}) even though the initial conditions are arbitrarily close to those of flocking states.
\end{remark}
For example, starting from the flocking solution $r_k(t)=(0, 0, t) $ for all $t$ and $k$, given $\epsilon >0$ small, consider the perturbed initial conditions $r_{k,0}= (\epsilon \cos \theta_k, \epsilon \sin \theta_k, 0)$ and $\dot r_{k,0}= (-\epsilon \sin \theta_k, \epsilon \cos \theta_k,\sqrt{1-\epsilon^2})$ where $\theta_k$ are angles such that
$ \displaystyle \sum _{k=1}^n e^{i \theta k} =0.$ Note that for these perturbations $R(0)= (0,0,0)$ and $V(0)=(0,0, \sqrt{1-\epsilon^2}).$ It is immediate to verify that the particles in the $\R^3$ version of (\ref{Main}) move with unit speed on helix-shaped trajectories  according to
$$ r_k(t)= (\epsilon \cos (t+\theta_k), \epsilon \sin (t+\theta_k), t \sqrt{1-\epsilon^2}). $$
The perturbed trajectories have $R(t)= (0,0, t\sqrt{1-\epsilon^2})$ and therefore  $|r_k(t)-R(t)|= \epsilon$ and
$|\dot r_k(t)-V(t)|= \epsilon$ for all $t,$ failing the asymptotics (\ref{MainZeroDr}) and (\ref{MainZeroV}).

\vskip1cm

\section{Eliminating the Zero Eigenvalue(s)}\label{SectionMovingFrame}

One of the difficulties in working with the system (\ref{Main}) is the drift in the direction of motion, as illustrated in Figure \ref{Drift}.
For configurations having nonzero mean velocity ($V(t)\neq 0 $ for $t>0$), we  decouple the  direction of motion from the other variables of the system. We achieve this by using a frame that moves and rotates according to the center of mass, as illustrated in Figure \ref{Frame}. The new frame annihilates the directional drift and the rotation symmetry of the system, just as working with $r_k-R$ eliminated the translation invariance.  Since the rotation and translation invariance were associated with the multiplicity-three zero eigenvalue for the linearization of (\ref{Main}), the eigenvalues of the ensuing reframed dynamics will be shown to be nonzero (either equal to $\pm i $ or with negative real parts).

\begin{notation} Within this section, if $a=(a_1, a_2)$ and $b=(b_1,b_2)$  are vectors in $\R^2$ with $b \neq 0,$ define the vector $a^\bot$ and the scalars $\mbox{comp}_b a$ and $\mbox{ort}_b a$ to be:
\beq
\label{defineort}
a^\bot =(-a_2, a_1), \; \; \mbox{comp}_b a= \frac{a^T b}{|b|},
 \; \; \mbox{ort}_b a= \frac{a^T b^\bot}{|b|}.
 \eeq
 \end{notation}

 \noindent
Note that $\displaystyle (a^\bot)^\bot =-a, \; |a|^2=|a^\bot|^2=(\mbox{comp}_b a )^2+(\mbox{ort}_b a)^2$ and $\displaystyle \mbox{comp}_b b=|b|.$ Also:
\beq
\label{a_from_comp}
a= (\mbox{comp}_b a) \frac{b}{|b|} + (\mbox{ort}_b a)\frac{b^\bot}{|b|}.
\eeq

\begin{proposition}
Let $I$ be in interval and $a, b : I\to \R^2$  be differentiable functions, $b \neq 0.$ Then $\mbox{comp}_b a$ and $ \mbox{ort}_b a $ are differentiable and
\beq
\label{derivativesort}
\begin{array}{cl}
\frac{d}{dt} \left ( \mbox{comp}_b a \right ) &= \mbox{comp}_b \dot a + \frac{1}{|b|} ( \mbox{ort}_b a) ( \mbox{ort}_b \dot b) \\
\frac{d}{dt} \left( \mbox{ort}_b a \right ) & = \mbox{ort}_b \dot a - \frac{1}{|b|} ( \mbox{comp}_b a) ( \mbox{ort}_b \dot b).
\end{array}
\eeq

\end{proposition}
\begin{proof}

From (\ref{defineort}) and the fact that the derivative and $\bot$ commute,
\beq
\label{product_rule}
\begin{array}{cl}
\frac{d}{dt} \left ( \mbox{comp}_b a \right ) & =
(\dot a) ^T\frac{b}{|b|}+ a^T \frac{d}{dt} \frac{b}{|b|} =  \mbox{comp}_b \dot a + a^T \frac{d}{dt} \frac{b}{|b|}\\
\frac{d}{dt} \left( \mbox{ort}_b a \right ) & =
(\dot a)^T \frac{b^\bot}{|b|}+ a^T \left ( \frac{d}{dt} \frac{b}{|b|} \right )^\bot=
\mbox{ort}_b \dot a + a^T \left ( \frac{d}{dt} \frac{b}{|b|} \right )^\bot.
\end{array}
\eeq
From
$\displaystyle
\frac{d}{dt} \frac{b}{|b|} = \frac{1}{|b|}\dot b -  \frac{b^T \dot b  }{|b|^3}  b= \frac{1}{|b|}\left ((\mbox{comp}_b \dot b) \frac{b}{|b|} + (\mbox{ort}_b \dot b)\frac{b^\bot}{|b|}  \right)- \mbox{comp}_b \dot b \; \frac{b}{|b|^2}$
\newline
we get
\beq
\frac{d}{dt} \frac{b}{|b|} = \frac{1}{|b|^2} ( \mbox{ort}_b \dot b ) \;  b^\bot
\label{bversor}
\eeq
Substituting (\ref{bversor}) and $\displaystyle (b^\bot)^\bot =-b$  into (\ref{product_rule}) yields (\ref{derivativesort}).

\end{proof}

\begin{notation} For the system (\ref{Main}) with $V(t)\neq 0$, define the scalar functions of time $x_k, y_k, u_k, w_k$ and $m, s$ to be
\beq
\begin{array}{lcl}
x_k = \mbox{comp} _V (r_k-R) &\mbox{ and } &y_k= \mbox{ort} _V (r_k-R)\\
u_k= \mbox{comp} _V (\dot r_k)& & w_k= \mbox{ort} _V (\dot r_k)\\
s = \mbox{comp} _V (\dot V)& &  m= \frac{1}{|V|} \mbox{ort} _V (\dot V).
\end{array}
\label{newxy}
\eeq
\end{notation}

\noindent
The scalars $x_1, \dots w_n $ are the coordinates of the agents relative to a frame that moves along the center of mass, and rotates according to the direction of the mean velocity, as shown in Figure \ref{Frame}. $s$ and $m$ are introduced to streamline notations -- they can be expressed in terms of $x_1, \dots w_n $, as in (\ref{mandS1}) and (\ref{mandS2}).

\begin{figure}[h!]
\centering
\includegraphics[width=.7\textwidth]{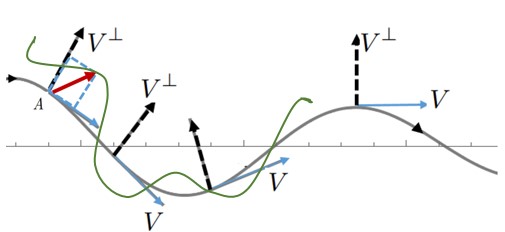}
\caption{ The moving frame used  to define $x_1, \dots w_n.$ The solid black curve represents the location of the center of mass, $R$; tangent to it are the velocity vectors $V$. The shorter curve illustrates the position of agent $1.$ The point $A$ shows center of mass at some time $t;$ at that time the displacement vector of the first agent is $r_1(t)-R(t),$  shown as the red diagonal vector. The sides of the dotted rectangle represent $x_1(t)$ and $y_1(t)$.}
\label{Frame}
\end{figure}

Note that $\displaystyle \dot V= s \frac{V}{|V|} +  m V^\bot $ thus  $m(t)$ captures how fast the direction of $V$ is changing {\it relative to the speed},  where $s(t)$ captures the changes in the speed of the center of mass.
The geometrical meanings of $m$ and $s$ are apparent if we represent the velocity $V$ in polar coordinates,
$V(t)=|V| (\cos \Theta (t), \sin \Theta (t)). $ Then
$$ \dot V= \frac{d |V|}{dt}(\cos \Theta (t), \sin \Theta (t))+ \frac{d\Theta}{dt}|V| ( -\sin \Theta (t)), \cos \Theta (t))=
\frac{d |V|}{dt} \frac{V}{|V|}+ \frac{d\Theta}{dt} V^\bot.$$
On the other hand $V= s \frac{V}{|V|}+ m V^\bot, $ therefore
$ \displaystyle s=\frac{d|V|}{dt} \; \mbox{ and } \; m=\frac{d\Theta}{dt}.$

From (\ref{MainAcc}) and (\ref{a_from_comp})  we get
$\displaystyle \dot V= \frac{1}{n}\sum_{k=1}^n (1-u_k^2-w_k^2)(u_k \frac{V}{|V|}+w_k\frac{V^\bot }{|V|}) .$  Collecting the tangential and normal components of the acceleration $\dot V$ as in (\ref{newxy}) yields
\beq
\label{mandS1}
s=\frac{1}{n}\sum_{k=1}^n (1-u_k^2-w_k^2)u_k = \frac{d|V|}{dt}\; \mbox{ and }
\eeq
\beq
\label{mandS2}
  m=\frac{1}{|V|} \frac{1}{n}\sum_{k=1}^n (1-u_k^2-w_k^2)w_k=\frac{d\Theta}{dt}.
\eeq
Before reformulating (\ref{Main}) in terms of  $x_1, \dots w_n$, we point out that $x_n$ will not be used as state variable, since
$\displaystyle
x_n=-\sum_{k=1}^{n-1} x_k,$ which follows from $\displaystyle \sum_{k=1}^n (r_k-R)=0 .$ The latter and $\displaystyle \frac{1}{n}\sum_k^n \dot r_k =V$ imply:
\beq
\sum_{k=1}^n x_k= \sum_{k=1}^n y_k=0,  \mbox{ and }  \sum_{k=1}^n w_k=0, \; \;  \frac{1}{n}\sum_k^n u_k=|V|.
\label{zerosums}
\eeq
Unlike $x_n,$ we keep $w_n$ and $y_n$ as state variables because eliminating them would hinder downstream computations (by contributing  off-pattern terms to the mean-field variables $V$ and $\dot V$ ). 

\begin{proposition}
The differential equations governing the $(4n-1)$ dynamical system with unknowns $w_1, \dots w_{n}$,  $y_1, \dots y_{n}$, $ x_1, \dots x_{n-1}$ and $u_1, \dots u_{n}$ are
\beq
\label{State}
\begin{array}{ll}
\dot w_k= (1-u_k^2-w_k^2)w_k- y_k -mu_k, \; &  k=1..n \\
\dot y_k= w_k-m x_k, \; &  k=1..n-1 \\
\dot y_n= w_n-m x_n= w_n-m ( -\sum_{k=1}^{n-1} x_k ) , &  \\
\dot x_k = u_k-\frac{1}{n}\sum_{k=1}^n u_k + m y_k,  \; &  k=1..n-1\\
\dot u_k = (1-u_k^2-w_k^2)u_k-x_k + m w_k ,  \; &  k=1.. n-1\\
\dot u_n = (1-u_n^2-w_n^2)u_n+\sum_{l=1}^{n-1}x_l + m w_n &
\end{array}
\eeq
where $m$ was defined in (\ref{mandS2}) as
$\displaystyle  m= \frac{1}{\frac{1}{n}\sum_{k=1}^n u_k} \frac{1}{n}\sum_{k=1}^n (1-u_k^2-w_k^2)w_k.$
\end{proposition}

\begin{proof}
Apply the results of (\ref{derivativesort}) for $ \displaystyle b=V, $  with $ \mbox{comp} _V (\dot V)=s,$ and
$ \mbox{ort} _V (\dot V)=m|V|.$ We get that for a differentiable function $a$
\beq
\label{derivprojected}
\begin{array}{ll}
\frac{d}{dt} \left ( \mbox{comp}_V a \right )= & \mbox{comp}_V \dot a + m \; \mbox{ort}_V a\\
\frac{d}{dt} \left ( \mbox{ort}_V a \right )= & \mbox{ort}_V \dot a -m \; \mbox{comp}_V a.
\end{array}
\eeq
Use (\ref{derivprojected})  for
$a=r_k-R $; recall that $\displaystyle  \mbox{comp}_V (r_k-R ) = x_k$ and
$\displaystyle  \mbox{ort}_V  (r_k-R )  = y_k.$ Moreover $\dot a= \dot r_k -V$ and thus $\dot a$ has components
$ \displaystyle \mbox{comp}_V \dot a= u_k-|V|$ and
$ \displaystyle \mbox{ort}_V \dot a= w_k.$ We get
\beq
\label{NewMain1}
\begin{array}{ll}
\dot x_k = u_k-|V|+ m y_k\\
\dot y_k=w_k-m x_k.
\end{array}
\eeq
Use (\ref{derivprojected}) again  for $ a=\dot r_k .$ Note that $\dot a=(1-u_k^2-w_k^2) \dot r_k-(r_k-R)$ so it has components
$ \displaystyle \mbox{comp}_V \dot a= (1-u_k^2-w_k^2)u_k-x_k$ and
$ \displaystyle \mbox{ort}_V \dot a= (1-u_k^2-w_k^2)w_k- y_k.$
We get
\beq
\label{NewMain2}
\begin{array}{lll}
\dot u_k = &(1-u_k^2-w_k^2)u_k-x_k + m w_k  &  k=1\dots n\\
\dot w_k= & (1-u_k^2-w_k^2)w_k- y_k -mu_k  &  k=1\dots n.
\end{array}
\eeq

\end{proof}

\begin{remark} The dynamical system $(\ref{State})$ is reflection invariant, in the sense that if $(w,y,x,u)$ is a solution for $(\ref{State})$, then $(-w, -y, x, u)$ is also a solution.

Let $H$ denote the hyperplane orthogonal to $\onen$ in $\R^n.$
The subspace $H\times H\times \R^{2n-1}$ is invariant under (\ref{State}).
\end{remark}
The reflection invariance holds since the functions  $mx$ and $mu$ are odd functions of $(w,y)$ and even in $(x,u)$ and thus the vector field given from $(\ref{State})$ has its $(\dot w, \dot y)$ components given by  odd function of $(w,y)$ and even in $(x,u),$ whereas the $(\dot x, \dot u)$ components of $F$ are even functions of $(w,y) $ and $(x,u).$
The invariance of $H\times H\times \R^{2n-1}$ follows from
 $\displaystyle \frac{1}{n} \sum_{k=1}^n \dot w_k = \frac{1}{n} \sum_{k=1}^n y_k $ and  $\displaystyle \frac{1}{n} \sum_{k=1}^n \dot y_k = \frac{1}{n} \sum_{k=1}^n w_k .$

\begin{remark} Any translating state $r_k (t)= (c_{0,1}, c_{0,2})+ te^{i \Theta_{0}} $ for  $k=1..n$ of (\ref{Main}) corresponds to the fixed point  $w_k=0, y_k=0$ and $u_k=1$ for $k=1..n$  and $x_k=0$ for $k=1..n-1$ of (\ref{State}). 
\end{remark}

\begin{proposition} Let $Z_f$ denote the fixed point of (\ref{State}) with $w_k=0, y_k=0, x_l=0$ and $u_k=1$ (for $k=1..n$ and $l=1..n-1$). Denote by  $J$  the matrix of the linearization of (\ref{State}) at $Z_f.$ Then
$$\mbox{det}(\lambda I -J) =(\lambda ^2+1) ^n(\lambda+1)^{2(n-1)}(\lambda+2).$$
The center subspace $E^c$ of $J$ has dimension $2n$ and is spanned by the variables $w, y.$ The stable subspace $E^s$ has dimension $2n-1$ and is spanned by the variables $x, u.$ There is no unstable subspace:
$ \displaystyle \R^{4n-1}= E^c\oplus E^s=\R^{2n}\oplus\R^{2n-1} .$

\end{proposition}
\begin{proof}
 Note that at $Z_f$ the speed is
 $|V|= \frac{1}{n}\sum_{k=1}^n u_k=1.$ Near $Z_f$ all the terms $(1-u_k^2-w_k^2)w_k$ are quadratic or smaller
  since both $(1-u_k^2-w_k^2)$ and $w_k$ are zero at $Z_f,$ therefore  $(1-u_k^2-w_k^2)w_k$ have zero gradients. From
  $\displaystyle m= \frac{1}{|V|} \frac{1}{n}\sum_{k=1}^n (1-u_k^2-w_k^2)w_k$ we conclude that $m$ and the terms $m u_k, mx_k, my_k$ and $mw_k$ vanish and have zero gradient at $Z_f$ as well.
  The nonlinear term $(u_k -u_k^3- w_k^2 u_k)$ that is part of $\dot u_k$ is linearly approximated by $(-2)(u_k-1)$ at $u_k=1, w_k=0.$
  Finally, note that the pattern for $\displaystyle \frac{\partial \dot u_n}{\partial x_l} $ is different than that for the partial derivatives of $\dot u_k$ and $k=1..n-1$ so when using block notations for the Jacobian matrix, we have to separate away the last $u$ component from the rest.

The Jacobian matrix (\ref{State}) at $Z_f $ is
\beq
\label{Jacobian}
J= \left [
\begin{array}{ll|lll}
\mathbb{O}_n & -I_{n}& \mathbb{O}& \mathbb{O} &0_{n,1}\\
I_{n} & \mathbb{O}_n& \mathbb{O}&\mathbb{O} &0_{n,1} \\ \hline
&&&&\\
\mathbb{O}& \mathbb{O}& \mathbb{O_{n-1}} & B_{n-1} & \frac{-1}{n}\onenn\\
\mathbb{O} &\mathbb{O} & -\Inn &  -2\Inn & 0_{n-1,1}\\
0_{1, n} &0_{1, n}  & \onenn^T & 0_{1, n-1} &-2
\end{array}
\right ]
\eeq
where $B_{n-1}$ is the square matrix $ \Inn-\frac{1}{n}\mathbb{1}_{n-1, n-1}$;
 on the left of the separating line $\mathbb{O}$ means $ \mathbb{O}_{n-1, n}$ and on  the right of the separating line
 $\mathbb{O}$ means $\mathbb{O}_{n, n-1}$.
Note that $J$ is already in a block diagonal form.
The first block, of size $2n$, corresponding to the state variables $w$ and $y$, has characteristic polynomial $\displaystyle (\lambda ^2+1) ^n.$ The second block, of size $(2n-1)$, corresponding to the state variables $x$ and $u$ has characteristic polynomial
$\displaystyle (\lambda+1)^{2n-2}(\lambda+2).$ One can check that by performing the following operations: first subtract the last column from each column in the middle block, then add all the rows of the middle block to the last row (the identity matrix $I$ being $\Inn$):

\bdima
\left |
\begin{array}{lll}
\lambda I  &-I+\frac{1}{n}\mathbb{1} & \frac{1}{n}\onenn\\
 I&  (\lambda+2)I & 0_{n-1,1}\\
 -\onenn^T & 0_{1, n-1} &\lambda+2
\end{array}
\right |=
\left |
\begin{array}{lll}
\lambda I  &-I& \frac{1}{n}\onenn\\
 I&  (\lambda+2)I & 0_{n-1,1}\\
 -\onenn^T & -(\lambda+2)\onenn^T &\lambda+2
\end{array}
\right |=
\edima
\bdima=
\left |
\begin{array}{lll}
\lambda I  &-I& \frac{1}{n}\onenn\\
 I&  (\lambda+2)I & 0_{n-1,1}\\
 0 & 0 &\lambda+2
\end{array}
\right |= (\lambda+2)(\lambda^2+2\lambda+1)^{n-1}.
\edima
\end{proof}

\begin{notation}
Given an initial condition vector $Z_0$ in   $\R^{4n-1}$ denote by
$Z=\psi_t(Z_0)$ the solution to (\ref{State}). Given an initial condition vector $X_0$ in $\R^{4n}$ consisting of $r_k(0)$ and $v_k(0)$ for (\ref{Main}), denote by $\phi_t(X_0)$ the solution to (\ref{Main}).
\end{notation}

\begin{remark}
\label{recoverflow}
For the proof of the Main Theorem \ref{MainTh}, it is enough to show that the fixed point $Z_f=col(\zn,\zn,\mathbb{0}_{n-1}, \onen) $ is an asymptotic fixed point for (\ref{State}) on $H\times H \times \R^{2n-1}$ and that  the integral $\int _0^{\infty } |m| dt$ converges and it is small for solutions starting  near $Z_f$.
\end{remark}

\begin{proof}
The assumptions of Main Theorem \ref{MainTh} and the assertions (\ref{FirstMain}) and (\ref{MainZeroDr})  can be restated using $\psi_t =col( w, y, x, u)$ alone, since they are norm-based and
\bdima
|r_{k_1}- r_{k_2}|^2=|(r_{k_1}-R)-( r_{k_2}-R)|^2= (x_{k_1}- x_{k_2})^2+ (y_{k_1}- y_{k_2})^2
\edima
and
\bdima
|V|= \frac{1}{n}\sum_{k=1}^n u_k \; \mbox{ and } \; |\dot r_{k}-V|^2= (u_k-|V|)^2+w_k^2.
\edima
This turns the assertions (\ref{FirstMain}) and (\ref{MainZeroDr}) into statements about the asymptotic stability of the flow $\psi_t$ at $Z_f.$

For the assertion (\ref{MainZeroV}) we need to examine how, given a solution  $\psi_t(Z_0)$ in the $(4n-3)$ dimensional space $H\times H \times \R^{2n-1}$, we can identify the trajectories $\phi_t$ in $\R^{4n}$ that in the center-mass frame (\ref{newxy})
lead back to $\psi_t(Z_0)$ . The trajectories $\phi_t$ can only be determined up to a 3-dimensional parameter $c_0$, interpreted as a location $(c_{0,1}, c_{0,2})$ in $\R^2$ together with an angle $c_{0,3}=\Theta_0.$

For  $\psi_t=(x_1,  \dots u_n)$ in $H\times H \times \R^{2n-1}$, define $\displaystyle \Theta(t)=\Theta_0+ \int _0^t m\; d\tau$ where  $\displaystyle m=\frac{1}{|V|} \frac{1}{n}\sum_{k=1}^n (1-u_k^2-w_k^2)w_k $
(see (\ref{mandS2})). Note that $|\Theta(t)-\Theta_0|\leq \int _0^{\infty } | m| dt.$

  Rebuild $V$ using $V=|V|(\cos \Theta(t), \sin \Theta(t))$ and from it find
  \newline
$\displaystyle R(t)= (c_{0,1}, c_{0,2}) + \int _0^t V dt.$ Note that  $\Theta_0$ captures the indeterminacy due to the rotational invariance of (\ref{Main}) and $(c_{0,1}, c_{0,2})$ captures the indeterminacy due to the translation invariance of (\ref{Main}).
Finally, we can recover the  directions of the vectors from $\phi_t$ using
\beq
 r_k= R(t)+ x_k  \frac{V}{|V|}+y_k\frac{V^\bot }{|V|}\; \; \mbox{ and }\;  \dot r_k= u_k  \frac{V}{|V|}+w_k\frac{V^\bot }{|V|}.
 \label{recover}
 \eeq
 Once the convergence of $\lim _{t\to \infty} |V(t)|= 1$ is established, assertion (\ref{MainZeroV}) reduces to proving the convergence  of the angle $\Theta$, which would follow from
 the convergence of $\int _0^{\infty } |m| dt.$
\end{proof}

\begin{remark}\label{analytic}  Let $\mathcal U$ be the open subset of $\R^{4n-1}$ where $\frac{1}{n}\sum_{k=1}^n u_k>0 .$ Let $t_0, Z_0$ be such that $\psi_{t_0}(Z_0) \in \mathcal U.$ Then there exists a time interval $(t_0-\epsilon, t_0+\epsilon)$ such that the function $t \to \psi_t(Z_0) $ is analytic on that interval.
\end{remark}
This holds since the vector field (\ref{State}) is defined by algebraic operations on $m, w, y, x, u.$ The condition that $\psi_{t_0}(Z_0) \in \mathcal U$ ensures that the fraction
 $\displaystyle\frac{1}{|V|}= \frac{n}{\sum u_k} $ used to define $m$ in terms of $w, y, x,u$
 (see (\ref{mandS2})) is real analytic in a neighborhood of $\psi_{t_0}(Z_0).$

\subsection{The Polar Angles in the $(w_k, y_k)$ Planes}

In this subsection we are only interested in solutions $\psi_t (Z_0)$ with $t\geq 0$  that lie in the set $\mathcal U$ where  $\frac{1}{n}\sum_{k=1}^n u_k>0.$  
Consider the projection of $\psi_t (Z_0)$ on the plane $(w_k, y_k);$ denote by $\gamma _k$ the curve with parametrization $(w_k(t), y_k(t)).$

The upper block of $J$ shows that each $(w_k, y_k)$  plane is an eigenspace for $\pm i.$
It is natural to recast the dynamics in each of these planes using polar coordinates, $w_k+iy_k= a_k e^{i \theta _k} $('a' is for amplitude), but we run into a technical difficulty: we would like to find  \emph{continuous } functions $\theta _k(t)$ to represent $\gamma _k$ as $\displaystyle \gamma_k= \sqrt{w_k(t)^2+ y_k(t)^2} e^{i\theta_k (t)}.$ If $\gamma_k(t) \neq 0  $ for all $t,$ then a smooth polar angle  $\theta _k(t)$ exists. However, the property $\gamma_k=0$ is not invariant under the flow: the origin, $a_1=\dots=a_n=0$ is a fixed point, but points such as $(a_1, a_2, \dots 0, \dots a_n)$ are not, meaning that $\gamma _k $ could go {\em through} $(0,0)$ and in doing so it could enter then emerge at different polar angles, forcing a jump in $\theta_k$ (see Figure \ref{NotPolar}).

\begin{figure}[h!]
\centering
\includegraphics[width=.6\textwidth]{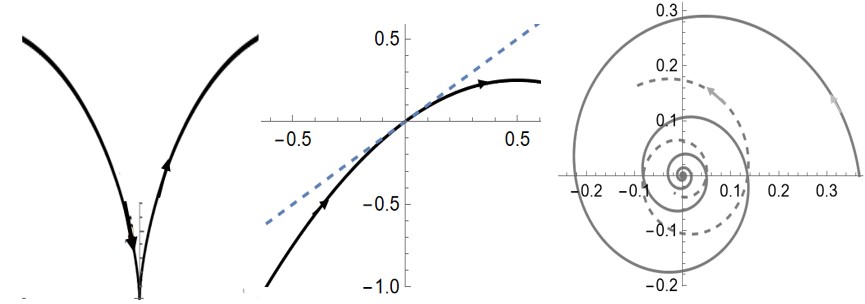}
\caption{ For curves with smooth parametrization a continuous polar angle $\t(t)$ may not exist if they go through the origin, as the illustrated curves do (when $t=0$).  Left: the cycloid  $(t -\sin t, 1 - \cos t)$ has $\t(0^-)=\pi/2=\t(0^+)$ so its polar angle is continuous. Middle: the parabola $(t, t-t^2)$ has its polar angle jump from $\t(0^-)=-3\pi/4$ to $ \t(0^+)= \pi/4.$ Right: the smooth curve $e^{-1/|t|} (\cos(1/t), \sin(1/t))$, shown as a solid curve for $t>0$ and dotted curve for $t<0.$ Its polar angle has $\t(0^+)=\infty,$ so it admits no continuous extension.}
\label{NotPolar}
\end{figure}
\begin{proposition}
\label{PieceAngle}
 (i) If $k, t_0$ and $\epsilon$ are such that $\gamma_{k} (t_0)=0$ with  $\gamma_{k}(t)\neq 0 $ for $0<|t-t_0|<\epsilon$ , then there exits a continuous function $\phi_k (t)$ such that one of the following is true
$$ \gamma_{k}(t)= |\gamma_{k}(t)|e^{i \phi_k(t)} \; \mbox{ for all } t \mbox{ with }  \; 0<|t-t_0|<\epsilon \; \; \; \; \; \; \; \mbox{ or}$$
\bdima
\gamma_{k}(t)= |\gamma_{k}|e^{i \phi_k(t)} \; \mbox{ for } t<t_0 \;  \;  { and }\; \;
\gamma_{k}(t)= |\gamma_{k}|e^{i \phi_k(t)+\pi i} \; \mbox{ for } t_0<t
\edima
(ii) Let $\gamma _k$ be such that the solution to $(\ref{State})$ remains in $\mathcal U$ for all $t>0,$ with $\gamma_k$ not identically zero. Then there exists a polar angle $\theta_k(t)$
such that $\displaystyle \gamma_k= |\gamma_k(t)|e^{i \theta_k(t)} $ and such that
the functions $\sin(2\theta_k(t))$ and $\cos(2\theta_k(t))$ are continuous. If $P$ is any continuously differentiable, periodic function with period $\pi$, the function $P(\theta _k (t))$ is continuous and piecewise differentiable.
\end{proposition}

\begin{proof}
  (i)
  We will use the analyticity of $w_k$ and $y_k$  and the injectivity of the tangent modulo $\pi$ ( $\tan \beta_1=\tan \beta_2$ implies $\beta_1-\beta_2 =0 \; \mbox{mod }\pi$).

   A polar angle $\alpha $ is well defined and continuous on $(t_0-\epsilon, t_0)$ and on $(t_0, t_0+\epsilon),$ due to $\gamma _k\neq 0$. We need to show that at $t_0$ the angle has sided limits that are equal or that differ by a multiple of $\pi.$

  If one of $w_k$ or $y_k$ is identically zero, $\gamma _k$ is along the coordinate axis; its polar angle is either $\{0, \pi\}$ or $\pm \pi/2.$ If neither $w_k$ nor $y_k$ are identically zero let
 $c_1(t-t_0)^{d_1}$ be the first nonzero term in the Taylor series for $w_k$ and let $c_2(t-t_0)^{d_2}$ be the first nonzero term for $y_k.$ If $d_1=d_2$ then $\gamma_k$ is tangent to the vector $(c_1, c_2)$ so the two sided limits of $\alpha $ exist and are in the set $\arctan(c_2/c_1)+\pi \mathbb Z.$ If $d_1<d_2$ then $\gamma_k$ is tangent to the $y_k$ axis and the polar angles have sided limits in the set $-\pi/2+\pi \mathbb Z.$ If $d_2<d_1$ then $\gamma_k$ is tangent to the $w_k$ axis and the polar angles have sided limits in the set $\pi \mathbb Z, $ so they differ by a multiple of $\pi.$
Let $p\pi= \alpha (t_0^+)-\alpha (t_0^-) $ and let $H$ be the Heaviside function. Then $\phi _k (t)=\alpha (t)-p\pi H(t-t_0)$ is continuous with $\gamma_{k_0}(t)= \pm |\gamma_{k_0}(t)|e^{i \phi_k(t)}.$

For (ii): there is nothing to prove if $\gamma_k$ is nonzero. Assume that $\gamma_k= 0$ at some positive times.
Due to the analyticity of $\gamma _k$ the set of times when $\gamma_k= 0$ is finite or consists of times  $0\leq t_1<t_2< \dots $ going to infinity. By (i), the polar angles corresponding to the intervals $(0, t_1),\; (t_1, t_2), \dots $ have well defined  side limits that differ by $0 \pi$ or $\pi$ at $t_1, t_2, \dots, $ meaning that the double angles coincide modulo $2\pi.$ This implies that $\sin(2\phi_k(t))$ and $\cos(2\phi_k(t))$ are continuous. The continuity properties for $P$ follow since $P$ is the sum of a Fourier series that converges in the sup norm.
\end{proof}

\begin{notation}
\label{reftheta}
(i) For the components $ w_k(t), y_k(t)$ of a solution to $(\ref{State})$ denote by
$a_k(t)=\sqrt{w_k^2+y_k^2} $ and by $\theta_k (t)$ the usual continuous polar coordinates if $(w_k(t), y_k(t)) \neq (0,0)$;  let $\theta_k(t)$ be the piecewise smooth polar angle from Remark \ref{PieceAngle}, part (ii) if $ (w_k(t), y_k(t))$ is not the identically zero function, and let $a_k=0, \theta_k(t)=t$ if $ (w_k(t), y_k(t))=(0,0)$ for all $t.$
Thus
\beq
\label{deftheta}
w_k= a_k \cos \theta_k, \; y_k= a_k \sin \theta_k.
\eeq

(ii) Denote by $a_{Max}$ or simply $a_M$ the Lipschitz continuous function
\beq
\label{amax}
a_M(t)=\max\{ a_k(t), k=1, .. n\}.
\eeq
\end{notation}
Note that for all smooth functions $P$ of period $\pi$  the functions $P(\theta _k(t))$ are continuous and piecewise smooth.

We conclude this section by  noting that although we will only need the polar coordinates representation  for the trajectories confined to the central manifold, we had to introduce $\theta_k$  while working with analytic functions of $t.$ Central manifolds and their flows are not necessarily $C^\infty,$  let alone analytic.

\section{The Dynamics on The Central Manifold}\label{Section_CentralManifold}

We shift coordinates in order to bring the fixed point $Z_f$ at the origin of $\R^{4n-1}: $
\begin{notation} Let $z_k=u_k-1$ for $k=1..n$ and let $\bar z=\frac{1}{n}\sum\limits_{k=1}^n z_k $ (thus $\bar z=|V|-1).$
\newline
Let $F=F(w, y,x, z)$ denote the vector field in $\R^{4n-1}$ given by
\beq
\label{State0}
F =
\left [
\begin{array}{l}
(-2z_k-z_k^2-w_k^2)w_k- y_k -m(1+z_k)   \\
  w_k-m x_k   \\
  \hline
 z_k-\frac{1}{n}\sum \limits_{k=1}^n z_l + m y_k  \\
 (-2z_k-z_k^2-w_k^2)(1+z_k)-x_k + m w_k  \\
 (-2z_n-z_n^2-w_n^2)(1+z_n)+\sum \limits_{l=1}^{n-1}x_l +m w_n
\end{array}
\right ].
\eeq
where  $\displaystyle m = \frac{1}{|V|}\frac{1}{n}\sum\limits_{l=1}^n (1-u_l^2-w_l^2)w_l =
\frac{1}{1+\frac{1}{n}\sum\limits_{l=1}^n z_l}\frac{1}{n}\sum\limits_{l=1}^n (-2z_k-z_k^2-w_k^2)w_l $ , 
with the index $k$ being in $1..n$ for the top components, and $k$ in $1..n-1$ for the components below the line.

\end{notation}

\noindent
Note that we used  $(1-u_k^2-w_k^2)=(-2z_k-z_k^2-w_k^2).$
Also: $\overline z= |V|-1$ gives the rate of convergence of the mean field speed to one. We get
\beq
\label{mdefine}
m= \frac{1}{1+\overline z} \frac{1}{n} \sum_{k=1}^n (-2z_k-z_k^2-w_k^2)w_k.
\eeq
In the $(w, y,x, z)$ coordinates the $4n-1$ dynamical system (\ref{State}) becomes
$$ col(\dot w, \dot y, \dot x, \dot z)= F(w, y, x, z).$$
An alternate expression for  $\dot x_k$ is
$\displaystyle
\dot x_k =
 z_k-\overline z + m y_k .$

 We want to prove that the system (\ref{State0}) has the origin as an  asymptotically stable fixed point.
The Jacobian matrix $DF$ for (\ref{State0}) at the origin is the same as that of (\ref{State}) at the point $Z_f, $ given in (\ref{Jacobian}), meaning that the central subspace for $DF(0)$ is associated with the eigenvalues $\pm i$, has dimension $2n$,  and it consists of the points $col(w, y, 0, 0).$

\subsection{Approximation of the Map of The Central Manifold}

This subsection uses the  standard PDE technique from \cite{Carr} to find a Taylor approximation for the central manifold map of (\ref{State0}) and to obtain differential equations for the flow on the central manifold.

  All the calculations from this subsection and beyond describe the solutions of  (\ref{State0}) for as long as they stay within a small distance $\delta$ from the origin.  $\delta$ is assumed small enough so that the local central manifold is defined within the ball $B_\delta$,  with the central manifold map $h$ having continuous $5^{th}$ order derivatives.

 We use the customary big-O notation $\mco(|\xi|^q)$ to convey a function that has absolute value at most a positive multiple of $|\xi|^q$. For example both $|w|$ and $|y|$ are $\mco( a_M),$ since $a_M$ is the largest amplitude among $(w_k, y_k).$

\begin{notation}
For a positive $q$ we use  $\mco_q$ to denote $\mco( a_M^q).$
\end{notation}

\begin{theorem} The flow on the central manifold $W^c$  of (\ref{State0}) is governed by the $2n$ dimensional system
\beq
\label{ODEwy}
\begin{array}{l}

\dot w_k=-y_k-m +w_k s_k+ w_k\mco_4 + m \mco_2, \\
\dot y_k= w_k+m\mco_2  \; \; \; \; \; \; \; \; \; \; \; \; \; \; \; \; \; \; \; \; \; \; \; \; \; \; \; \; \; \; \; \; \; \; \;  \; \; \; \; \mbox{ for } k=1..n, \; \mbox{ where } \\
\\
s_k=-\left( \frac{17}{25}w_k^2-\frac{12}{25}w_k y_k+ \frac{8}{25}y_k^2 \right ) +
 \left ( \frac{43}{200} \sww+\frac{1}{100} \swy +\frac{57}{200}\syy \right ) \\
 \\
 \sww= \frac{1}{n}\sum \limits_{k=1}^n w_k^2, \; \; \; \swy= \frac{1}{n}\sum \limits_{k=1}^n w_k y_k, \; \;
\syy= \frac{1}{n}\sum \limits_{k=1}^n y_k^2 \; \; \; \mbox{ and }\\
 m=  \frac{1}{n}\sum \limits_{k=1}^n s_kw_k +\mco_5=\mco _3.
\end{array}
\eeq
\end{theorem}

\begin{proof} Let $h$ denote the map whose graph is the central manifold, as in \cite{Carr}. In our notations,
 the central manifold is described by $\displaystyle col(x, z)=h(w,y).$ Since the projections onto the $(x, z)$ components of the vector  fields $ F(-w,-y,x,z)$ and $F(w, y, x, z)$ are equal, the map $h$ can be constructed to be even,
 $h(-w, -y)=h(w, y).$
 In particular the Taylor expansion of $h$ has no odd-power terms.

 For a scalar function $\phi=\phi(w, y)$ denote by $D_w \phi $ the $1\times n$ row vector of the partial derivatives with $w_1, \dots w_{n},$ and denote by $D_y \phi $ the row vector of the derivatives with $y_1, \dots y_{n}.$

We seek to produce an approximation
$\displaystyle \left[ \begin{array}{l} X\\
Z \end {array}\right ] $ of $h$ such as
$\displaystyle \left[ \begin{array}{l}  X
\\Z \end {array}\right ] -h= \mco _4.$ Given that $h$ has no third degree terms in its expansion, we seek quadratic polynomials such that
$\displaystyle \left[ \begin{array}{l}  X
\\Z \end {array}\right ] -h= \mco _4. $  We separate the terms of (\ref{State0}) that are smaller  than the needed precision. A priory  $x_k=\mco_2$ and $z_k=\mco_2.$ Their gradients are $\mco_1.$ We get
\beq
\begin{array}{l}
(-2z_k-z_k^2-w_k^2)=(1-u_k^2-w_k^2)= -2z_k-w_k^2 +\mco_4=\mco_2 \\
\\
\frac{1}{|V|}=\frac{1}{1+ \frac{1}{n}\sum\limits_{k=1}^n z_n} =1+\mco_2  \\
\\
m = \frac{1}{|V|}\frac{1}{n}\sum\limits_{k=1}^n (1-u_k^2-w_k^2)w_k=
\frac{1}{1+\mco_2}\frac{1}{n}\sum\limits_{k=1}^n \mco_2 \mco_1=
\mco_3 \\
\\
(-2z_k-z_k^2-w_k^2)(1+z_k)=(-2z_k-z_k^2-w_k^2)+\mco_4
\end{array}
\label{miso3}
\eeq
We conclude that
\beq
\label{State3}
\begin{array}{l}
\dot w_k=-y_k-2z_k w_k -w_k^3-m+\mco_4=   -y_k +\mco_3\\
\dot y_k=w_k- mx_k= w_k+ m\mco_2\\
\\
\dot x_k= z_k-\frac{1}{n}\sum \limits_{l=1}^n z_l +\mco_4\\
\dot z_k= -2z_k-x_k-w_k^2+\mco_4\\
\dot z_n= -2z_n+\sum\limits_{l=1}^{n-1}x_l-\left ( \sum\limits_{k=1}^{n-1}w_l \right )^2 +\mco_4
\end{array}
\eeq

Following the approximation technique from  \cite{Carr} Theorem 3, page 5,
based on the vector field (\ref{State3}), we need quadratic polynomials $X_k$ with $k=1..n-1$ and $Z_k$ with $k=1..n$ such that the differences $\delta^X , \delta^Z $ given below are $\mco_4.$
\bdima
\begin{array}{l}
\delta^X_k =[ D_w X_k\; D_y X_k]
\left [\begin{array}{r} -y +\mco_3\\
w+\mco_3
\end{array} \right ] - \left ( Z_k-\frac{1}{n}\sum \limits_{l=1}^n Z_l +\mco_4 \right )\\
\delta^Z_k =[ D_w Z_k\; D_y Z_k]\left [\begin{array}{r} -y +\mco_3\\
w+\mco_3
\end{array} \right ] - \left ( -2Z_k-X_k-w_k^2+\mco_4 \right ) \\
\delta^Z_n =[ D_w Z_n\; D_y Z_n]
\left [\begin{array}{r} -y +\mco_3\\
w+\mco_3
\end{array} \right ] - \left (-2Z_n+ \sum\limits_{l=1}^{n-1} X_l -(\sum \limits_{k=1}^{n-1}w_l)^2+\mco_4 \right )
\end{array}
\edima
For a more compact notation, denote by $L $ the linear differential operator
\newline $ Lf= -( D_w f)y+ (D_y f)w.$ Moving the error terms to the left, we get

\beq
\label{quadratic}
\begin{array}{l}
\delta^X_k +\mco_4 =L X_k -\left ( Z_k-\frac{1}{n}\sum \limits_{l=1}^n Z_l \right )\\
\delta^Z_k +\mco_4=L Z_k - \left ( -2Z_k-X_k-w_k^2 \right )\\
\delta^Z_n +\mco_4=L  Z_n- \left (-2Z_n+ \sum\limits_{l=1}^{n-1} X_l -(\sum \limits_{k=1}^{n-1}w_l)^2\right )
\end{array}
\eeq
We meet the $\mco_4$ requirement for the differences $\delta^X , \delta^Z $  if the right hand side of (\ref{quadratic}) is identically zero.
It is more expedient to use $ \overline Z $ defined as
$$ \overline Z= \frac{1}{n} \sum _{l=1}^n Z_l$$
rather than $Z_n$. To achieve $\mco_4$ precision we need that for $k=1..n-1$ the following hold:
\bdima
\begin{array}{l}
L X_k -( Z_k-\overline Z)=0\\
L Z_k +2Z_k+X_k=-w_k^2 \\
L \overline Z+2 \overline Z=-\frac{1}{n}\sum \limits_{k=1}^n w_k^2
\end{array}
\edima
Note that by subtracting the last equation from the middle row equations we get
\beq
\label{LZ}
\begin{array}{l}
L X_k -( Z_k-\overline Z)=0\\
L (Z_k-\overline Z) +2(Z_k-\overline Z) +X_k=-(w_k^2 -\frac{1}{n}\sum \limits_{k=1}^n w_k^2)\\
L \overline Z+2\overline Z=-\frac{1}{n}\sum \limits_{k=1}^n w_k^2
\end{array}
\eeq
Before proceeding to finding the quadratic polynomials $X_k, Z_k, \bar Z,$ we perform preliminary calculations to determine how $L$ acts on simple quadratic terms.
We compute $L w_k^2=-2w_k y_k, \; \;  L w_ky_k= w_k^2-y_k^2, \; \; L y_k^2=2w_ky_k.$ From these and
\beq
\label{sigmas}
\sww= \frac{1}{n}\sum_{k=1}^n w_k^2, \; \swy= \frac{1}{n}\sum_{k=1}^n w_k y_k,, \; \;
\syy= \frac{1}{n}\sum_{k=1}^n y_k^2.
\eeq
we get
\beq
\label{operatorL}
\begin{array}{cllll}
L\sww &= -2\swy   & \mbox{ and }&  L (w_k^2-\sww)&= -2(w_k y_k-\swy)  \\
L \swy &= \sww-\syy & &     L (w_k y_k -\swy)&=(w_k^2-\sww) -(y_k^2-\syy) \\
L \syy &= 2\swy & & L(y_k^2- \syy)&= 2(w_k y_k -\swy)\\
\end{array}
\eeq
Using vector notation we rewrite (\ref{operatorL}) as
\bdima
L \tayvec = \left [ \begin{array}{ccc} 0&-2&0\\
1&0&-1\\
0&2&0
\end{array} \right ] \tayvec
\edima
\beq
\label{matrixL}
L \dtayvec = \left [ \begin{array}{ccc} 0&-2&0\\
1&0&-1\\
0&2&0
\end{array} \right ] \dtayvec.
\eeq
Let $M_L$ denote the matrix of the operator, $\displaystyle M_L= \left [ \begin{array}{ccc} 0&-2&0\\
1&0&-1\\
0&2&0
\end{array} \right ] . $

\vskip5mm
\noindent
Motivated by the fact that the original system (\ref{Main}) had index-permutation invariance, and from the identities  $\sum \limits_{l=1}^n x_l=0$
and $\sum \limits_{l=1}^n ( z_l - \overline z) =0, $
we look for quadratic polynomials $X_k, (Z_k-\overline Z), \overline Z$ with similar features.
We want to find constants $c_1, \dots c_9$ such that for $k=1..n-1$

\bdima
\begin{array}{cl}
X_k& =c_1 (w_k^2-\sww)+c_2 (w_k y_k - \swy) +c_3 (y_k^2- \syy)\\
&\\
Z_k-\overline Z& = c_4 (w_k^2-\sww)+c_5 (w_k y_k-\swy)  +c_6 (y_k^2-\syy)\\
&\\
\overline Z&=c_7\sww+c_8 \swy +c_9\syy
\end{array}
\edima
Group the unknown  coefficients into the \emph{row}  vectors $c_X=[c_1, c_2, c_3]$,  $c_Z=[c_4, c_5, c_6]$ and
$\overline c= [c_7, c_8, c_9].$ We get that
\bdima
X_k= c_X\dtayvec, \; \; Z_k-\overline Z = c_Z \dtayvec, \; \; \overline Z=\overline c \tayvec.
\edima
From (\ref{matrixL}) we get
\bdima
L\;  X_k= c_X M_L \dtayvec, \; \; L( Z_k-\overline Z) = c_Z M_L \dtayvec, \mbox{ and }
\edima
\bdima
L\overline Z=\overline c M_L \tayvec.
\edima
Use  matrix notation to rewrite (\ref{LZ})  as:
\bdima
 ( c_X \; M_L- c_Z ) \dtayvec =0,
 \edima
 \bdima
  ( c_Z  M_L+2c_Z+ c_X)\dtayvec= \left [
\begin{array}{lll}
-1& 0 & 0
\end{array}
\right ]\dtayvec
\edima
and
\bdima
(\overline c \; M_L+2\overline c)\tayvec =\left [ \begin{array}{lll}
-1& 0 & 0
\end{array}
\right ]\tayvec
\edima
Use identification of coefficients in the latter equality: we require  $\displaystyle \overline c \; (M_L+2I_3)= [-1\; 0\; 0].$
The solution is $ \overline c =[ -\frac38, -\frac14, -\frac18. ]$ We conclude
$$\overline Z =-\frac38 \sww- \frac28\swy - \frac18 \syy.$$
By identification of coefficients we require that the 6-dimensional \emph{row} vector $[c_X, c_Z]$ solves
\bdima
[c_X \; c_Z ]
\left[
\begin{array}{l|c }
M_L& I_3\\
\hline\\
-I_3& M_L+2I_3
\end{array}
\right ]=[0\; 0\; 0\;\;  -1 \; 0 \; 0]
\edima
We get that $\displaystyle c_X= [-\frac{11}{25}, \; -\frac{4}{25}, \; -\frac{14}{25}]$ and
$\displaystyle c_z= [-\frac{4}{25}, \; -\frac{6}{25}, \; \frac{4}{25} ].$

\noindent
Recall that $Z_k=c_Z\dtayvec + \overline c \tayvec.$ We get that for $k=1..n$
\bdima
Z_k=-\frac{4}{25}w_k^2-\frac{6}{25}w_k y_k+ \frac{4}{25}y_k^2+
 \left ( -\frac{43}{200} \sww-\frac{1}{100} \swy -\frac{57}{200}\syy \right ).
\edima
Substitute $\displaystyle h=col[X, Z]+\mco_4$ into the differential equations (\ref{State3}). Up to $\mco_4,$
\bdima
-2z_k -w_k^2
\simeq-\left(\frac{17}{25}w_k^2-\frac{12}{25}w_k y_k+ \frac{8}{25}y_k^2 \right ) +
 \left ( \frac{43}{100} \sww+\frac{2}{100} \swy +\frac{57}{100}\syy \right ),
\edima
thus $-2z_k -w_k^2= s_k+\mco_4.$ To complete the approximation of $\dot w_k$ note that the exact equation  $\dot w_k = (-2z_k-z_k^2-w_k^2)w_k-y_k-m(1+z_k)  $  differs from $(-2z_k -w_k^2)w_k -y_k -m $ by
$z_k^2w_k-mz_k=w_k \mco_4+m \mco _2.$ The term $(-2z_k -w_k^2)w_k $ differs from  $s_k w_k $ by $\mco _4 w_k.$ Overall,
\bdima
\dot w_k =-y_k-m +w_k s_k+m\mco_2 +w_k \mco_4
\edima
and $\dot y_k= w_k+m\mco_2 $ per (\ref{State3}).
Substitute  $-2z_k -w_k^2= s_k+\mco_4, \; z_k^2 =\mco_4 $ and $\bar z=\mco_2$   in
$\displaystyle m= \frac{1}{1+\overline z} \frac{1}{n} \sum_{k=1}^n (-2z_k-z_k^2-w_k^2)w_k $ from (\ref{mdefine}) to  complete
(\ref{ODEwy}).

\end{proof}

\begin{remark}
On the central manifold, the mean speed $|V|$ is always under one, and
$$1-|V| \geq  C \frac{1}{n} \sum_{l=1}^n (w_l^2+y_l^2) + \mco_4 .$$
\end{remark}
Proof: Use  $\displaystyle 1-|V|=-\overline Z + \mco_4= \frac{1}{n} \sum_{l=1}^n  (\frac38 w_l^2+ \frac28 w_l y_l + \frac18 y_l^2 )$ and the fact that the quadratic form $ \displaystyle \frac38 a^2+ \frac28 ab + \frac18 b^2 $ is positive definite with minimum equal to $C=\frac{2-\sqrt 2}{8} (a^2+b^2). $

\subsection{Amplitude Variation on The Central Manifold}\label{AV}

 We consider trajectories $(w, y)$  in the subset $H\times H$ of the central manifold, for as long as they remain in the small ball $B_\delta$ where the  central manifold is  $C^5$. We use the polar coordinate representation  from (\ref{deftheta}) and Remark \ref{reftheta}.

Let $A, E$ be periodic functions of period $\pi$ defined as
\beq
\label{AE}
\begin{array}{l}
 A(\t)=-\cos ^2 \t \left [ \frac{17}{25}\cos^2 \t -\frac{12}{25}\sin\t \cos \t+ \frac{8}{25}\sin ^2 \t \right ] \; \; \; \mbox{ and }\\
 \\
E(\t)=  \frac{43}{100} \cos ^2 \t +\frac{2}{100} \sin \t \cos \t  +\frac{57}{100}\sin^2 \t.
\end{array}
\eeq

We get from the choice of $\t_k$ in $ (w_k, y_k) = ( a_k \cos \t_k , a_k \sin \t_k)$ that the functions $A(\t _k(t) )$  and $E(\t_k(t) )$ are continuous for all $t$, and piecewise $C^5.$ Moreover, non-differentiability of $A(\t _k)$ and $E(\t_k)$ can only happen for times $t_{oo}$ and indexes  $k_{oo}$ with
$w_{k_{oo}}(t_{oo})= y_{k_{oo}}(t_{oo})=0$. We abbreviate this excepted set of times  (this restriction) by $\mathcal E{oo}.$

\begin{proposition}
Using polar coordinates, the differential equations for $w_k, y_k $ become:
\beq
\begin{array}{l}
\dot a_k = -m\cos \t_k +a_k^3 A(\t_k) + a_k \cos^2 \t_k \frac{1}{n} \sum \limits_{l=1}^n a_l^2 E(\t_l)+\mco_4\\
 \\
\dot{\theta}_ k = 1+\frac{m}{a_k}\mco_0 + \mco_2, \; \; \; \; \mbox{ except for } \; \mathcal E{oo}.
\end{array}
\label{polar}
\eeq
The condition $\displaystyle \sum \limits _{k=1}^n w_k=0$ rewrites as $ \displaystyle \sum \limits _{k=1}^n a_k \cos \t_k=0.$
\end{proposition}
\begin{proof}

Use
$a_k \dot a_k = w_k \dot w_k + y_k \dot y_k= -mw_k + w_k^2 s_k + w_k \mco_4$ per (\ref{ODEwy}). Note that $\frac{w_k}{a_k} \mco_4$ is  $ \mco_4. $
The periodic functions $A, E$ allow us to rewrite $w_k^2 s_k$ as
\bdima
w_k^2 s_k= a_k^4 A(\t_k) + a_k^2\cos^2 \t_k \frac{1}{n} \sum _{l=1}^n a_l^2 E(\t_l),
\edima
and to conclude $\dot a_k = -m\cos \t_k +a_k^3 A(\t_k) + a_k \cos^2 \t_k \frac{1}{n} \sum _{l=1}^n a_l^2 E(\t_l) + \mco_4.$
Approximate $\dot{\t_k}$  from
$ \displaystyle \dot w_k = (-y_k +s_kw_k -m) + w_k \mco_4+m \mco _2$ and $\dot y_k = w_k+m\mco_2$:
\bdima
\dot{\theta}_ k= \frac{\dot y_k  w_k -\dot w_k y_k}{a_k^2} = 1+\frac{1}{a_k^2}(mw_k \mco_2-s_kw_ky_k+my_k- w_ky_k\mco_4-my_k\mco_2)
\edima
\bdima
=1+ \frac{m}{a_k}( \cos \t_k  \mco _2+ \sin \t_k + \sin \t_k \mco_2) -s_k \sin \t_k\cos \t_k  +\sin\t_k\cos\t_k \mco_4.
\edima
The factor following  $\frac{m}{a_k}$ has bounded terms, thus is $\mco_0$, and  $ \dot{\theta}_ k= 1+\frac{m}{a_k}\mco_0 + \mco_2.$

\end{proof}

Per (\ref{polar}), the amplitude $a_k$ is forced down by the negative term $a_k^3 A(\t_k)$ and amplified by the positive coupling terms $\cos^2 (\t_k) E(\t_l).$ To understand which of the two has the stronger effect, we examine their average impact over a cycle. Note that $A \leq 0$  and has average $\mu(A)=\frac{1}{\pi}\int_0^\pi A(\t) d\t= -\frac{59}{200}=-0.295. $  $E$ is positive and has average  $\mu(E)=0.5.$
 Estimating the impact of agent $l$ on the amplitude $a_k$ requires we investigate the products between  $\cos^2 (\t ) $ and all possible shifts $E(\t + \tau), $  as shown in Figure \ref{FigPeriodicA}. Considering all potential translations of functions, a.k.a. their hulls, is a common practice in dynamical systems  involving almost periodic functions that we would like to avoid. Subsection \ref{Subsection_TimeDependent} provides an alternate approach.

\begin{figure}[h!]
\centering
\includegraphics[width=.8\textwidth]{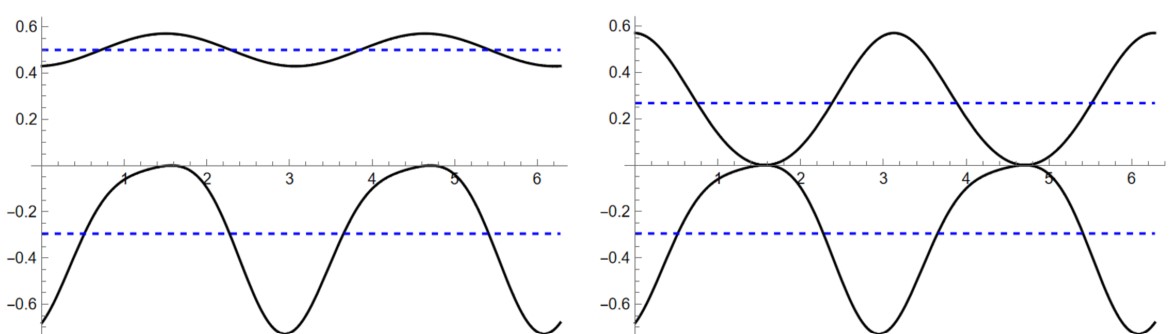}
\caption{
Left: A plot of the periodic functions $A\leq 0 $ and  $E> 0$, with their respective means, namely $\mu(A)=-0.295 $ and  $\mu(E)=0.5.$  Right: the functions $\displaystyle y= \cos^2 (\t) E(\t+\pi/2) $  of mean $0.2675$ 
and  $y=A(\t)$ of mean $-0.295 .$}
\label{FigPeriodicA}
\end{figure}

\subsection{The Stability of the Zero Solution for Non-autonomous Systems with Cubic Nonlinearities.}\label{Subsection_TimeDependent}

This subsection places the proof of stability for (\ref{polar}) from Section \ref{Section_Stability}
in a broader context. However,  the proof itself is independent of the content of this subsections, except for the calculations for the mean values of functions (\ref{means}),  Remark \ref{quotientrule} and Proposition \ref{FancyIneq}.

 We consider the systems (\ref{AlmostP1}) where the unknown $x=(x_1, \dots x_n)^T$ is vector-valued from $\R$ to $\R^n$, the coefficient functions $ b_k, c_k , d_k: \R \to \R$ are continuous and the remainder functions $R_k$ defined on $\R^n\times \R $ satisfy the
 $\mbox{\it{o}}_{3}$
 condition:
\bdima
\dot x_k = b_k(t)x_k^3  + c_k(t) x_k  \frac{1}{n} \sum \limits_{l=1}^n d_l(t) x_l^2+ R_{k}(x, t), \; \mbox{ for }  k=1 \dots n.
\edima

As explained in the introductory section, such systems are related to the dynamics of autonomous systems near fixed points whose Jacobian have purely imaginary roots
 $\omega = (\pm i \omega_1, \dots , \pm i \omega _n).$
  In most practical settings the coefficient functions are combinations of trigonometric functions of periods $2\pi/\omega_k, $ so they are almost periodic.\footnote{ If $\mathcal T$ denotes the set of all real-valued trigonometric polynomials (all linear  combination of $\cos( \lambda_k t)$ and $ \sin (\lambda _k t)$, with real $\lambda_1, \lambda _2, \dots $ ),  then a real-valued function $f$ is called almost periodic if there exists a sequence of trigonometric polynomials  $ T_k \in \mathcal T$ such that $T_k$ converges to $f$ in the sup norm on $\R.$
 Let $\mathcal A _{\R} $ denote the set of real-valued almost periodic functions. $\mathcal A_{\R}$ is closed under function multiplication and it is a closed linear subspace of the space $C_b(\R)$ of continuous and bounded functions
; in fact $\mathcal A_{\R}$ is the smallest subalgebra of $C_b(\R)$  containing all the continuous periodic functions (\cite{Corduneanu}). }

\begin{notation}
Define the mean  $\mu(f)$ of a function $f:[0, \infty) \to \R$ to be
\beq
  \mu(f)=\lim_{P\to \infty} \frac{1}{P}\int _{0}^P f(t) dt, \; \mbox{ whenever the limit exists.}
  \label{defMean}
  \eeq
Let  $\mathcal{BIM} $ denote the subspace of bounded, continuous functions that have a mean, and have bounded off-mean antiderivative:
 \beq
 \mathcal{BIM}= \{ f\in C_b([0, \infty))\;  | f \; \mbox{has mean }, \int  \left(  f- \mu(f) \right ) dt \in C_b([0, \infty))\}.
 \label{BIMspace}
\eeq

\end{notation}

If $f$ is a continuous, periodic function with period $T$ then $\mu(f)$ is the customary average value: $\mu(f)=\frac{1}{T}\int _0^T f(t) dt.$ Moreover, the antiderivative $\int (f(t)-\mu(f)) dt $ is periodic, thus  $f \in \mathcal{BIM}.$ Any linear combination of periodic functions belongs to $ \mathcal{BIM}.$ \footnote{Almost periodic functions have a mean. There exist zero-mean almost periodic functions with unbounded antiderivatives, so
 the set of almost periodic functions is not a subset of $ \mathcal{BIM}. $}

We can now state the stability result for (\ref{AlmostP1}); in keeping with earlier conventions, for a solution vector  $x=(x_1, \dots x_n)^T$ use $x_M=x_M(t)$ to denote the largest absolute value among its components:  $x_M(t)=\max_k |x_k(t)|.$

\begin{proposition}
\label{StableAlmost}
Consider the system (\ref{AlmostP1}), with continuous coefficient functions $b_k, c_k, d_k$ and 
reminders $R_k= \mbox{\it{o}}_{3}.$ Additionally, assume that $b_k$ and $(c_k+d_k)^2$ are in $\mathcal{BIM}.$

If $\mu(b_k)+ \frac{1}{n} \sum \limits_{l=1}^{n} \mu\left ( \frac14 (c_l+d_l)^2 \right ) < 0$ for all $k$ then the $x=\zn$ solution is asymptotically stable, with $x_M(t) \leq C x_M(0)/\sqrt t$ for large $t.$

\end{proposition}

\begin{proof}
Within this proof, use $\mathcal O_q$ for $\mathcal O \left( \; (\max_k |x_k|)^q \; \right ).$ Work in a small ball $B_\delta$ centered at the origin in $\R^n$.

The functions $b_k$ are in $\mathcal{BIM},$ so the antiderivatives $\int _0^t (b_k(\tau)-\mu(b_k))d\tau $ are bounded, (\ref{BIMspace}); by adding large enough constants, we obtain antiderivatives that are {\em positive} and bounded.
Let $B_k$ denote an antiderivative of $b_k-\mu (b_k)$ that has positive (and bounded) range.
Let $\mu_k= \frac14 \mu\left( (c_k(t)+d_k(t))^2 \right )  $ and let $C_k$ denote an antiderivative of $\frac14 (c_k+d_k)^2-\mu_k$ with positive (and bounded) range.

 Define the Lyapunov functions $W_k$ and $W$ as
\beq
\label{Wfunction}
W_k(t)= \frac{x_k^2(t)}{2x_k^2(t)\left (B_k(t))+ \frac{1}{n} \sum \limits_{l=1}^{n} C_l(t) \right )+1}, \; \mbox{ and } \;  W= \frac{1}{n} \sum \limits_{k=1}^{n} W_k.
\eeq

For as long as $x(t)$ is in the ball $B_\delta$ , the denominators of $W_k$  are  close to $1,$  so $W_k$ is approximately equal to $x_k^2.$ In particular $W$ and $x_M^2$ are only within a factor of $n$ from each other.

Before differentiating $W_k$ we take note of the following:
\begin{remark}
\label{quotientrule}
For $f, g$  positive functions, with $g, \dot g$ bounded, and  $f$ small,
\bdima
\frac{d}{dt} \frac{f}{2fg+1}= \frac{\dot f-2f^2\; \dot g}{(2fg+1)^2}= (\dot f-2f^2\; \dot g)(1 + \mco(|f|)).
\edima
\end{remark}

\noindent
To estimate $\dot W_k$ apply the latter with $f=x_k^2$ and $g=B_k + \frac{1}{n} \sum \limits_{l=1}^{n} C_l. $ We get
\bdima
\frac12 \dot W_k= \left [ x_k \dot x_k -  x_k ^4 \left (  \dot B_k +
\frac{1}{n} \sum \limits_{l=1}^{n}\dot C_l
 \right )   \right ] (1+\mco_2 )=
\edima
\bdima
b_k x_k^4  + c_k x_k^2 \left [\frac{1}{n} \sum \limits_{l=1}^n d_l x_l^2 \right ]+ x_k^4 \left ( -b_k+\mu(b_k)- \frac{1}{n} \sum \limits_{l=1}^{n} ( \frac14 (c_l+d_l)^2 -\mu_l ) \right )+ \mco_4\mco_2
\edima
\beq
\label{almostWkdot}
=x_k^4\left ( \mu(b_k)+\frac{1}{n} \sum \limits_{l=1}^{n}  \mu_l \right ) + c_k x_k^2 \left [\frac{1}{n} \sum \limits_{l=1}^n d_l x_l^2 \right ] - x_k^4\left [ \frac{1}{n} \sum \limits_{l=1}^{n}  \frac14 (c_l+d_l)^2 \right ] +\mco_6
\eeq
since the terms $\dot x_k x_k$  are $\mco_4$ and $\dot B_k, \dot  C_l$ are bounded.

Use the constant $K'$  for  $\max _k\{ \mu(b_k)+\frac{1}{n} \sum \limits_{l=1}^{n}  \mu_l  \}=-K'.$ By our assumptions $K'> 0$ and $\displaystyle x_k^4\left ( \mu(b_k)+\frac{1}{n} \sum \limits_{l=1}^{n}  \mu_l \right ) \leq -K' x_k^4.$ Substituting into (\ref{almostWkdot}), and averaging over $k$ in $1\dots n$ gives:

\bdima
\frac12 \dot W \leq -K' \left [ \frac{1}{n} \sum \limits_{k=1}^{n} x_k^4 \right ]+\mco_6+
\edima
\bdima
 + \left [ \frac{1}{n} \sum \limits_{k=1}^{n} c_k x_k^2 \right ]\left [\frac{1}{n} \sum \limits_{l=1}^n d_l x_l^2 \right ]- \left [ \frac{1}{n} \sum \limits_{k=1}^{n} x_k^4 \right ]\left [ \frac{1}{n} \sum \limits_{l=1}^{n}  ( \frac12 c_l+ \frac12 d_l)^2 \right ].
\edima
Because
$W^2$ and $\displaystyle \frac{1}{n} \sum \limits_{k=1}^{n} x_k^4 $ are comparable, and larger than $\mco_6, $ there exists $K>0$ such that
\beq
 \dot W \leq -K W^2  + 2\left [ \frac{1}{n} \sum \limits_{k=1}^{n} c_k x_k^2 \right ]\left [\frac{1}{n} \sum \limits_{l=1}^n d_l x_l^2 \right ]- 2\left [ \frac{1}{n} \sum \limits_{k=1}^{n} x_k^4 \right ]\left [ \frac{1}{n} \sum \limits_{l=1}^{n}  ( \frac12 c_l+ \frac12 d_l)^2 \right ].
\label{almostWdot}
\eeq

We will need the result of the following Proposition

\begin{proposition}
\label{FancyIneq}
Let $d$ be a positive integer, and let $p_k, q_k, r_k, \; k=1\dots d $ be scalars in $[0, \infty ).$  Then
\bdima
\left  [ \sum \limits_{k=1}^d p_k q_k \right ]\left  [ \sum \limits_{k=1}^d p_k r_k \right ] \leq
\left  [ \sum \limits_{k=1}^d p_k^2 \right ]\left  [ \sum \limits_{k=1}^d \left ( \frac12  q_k + \frac12 r_k \right ) ^2\right ].
\edima
\end{proposition}

\begin{proof}
Let $\displaystyle a=  \sum \limits_{k=1}^d p_k q_k $ and let $\displaystyle b= \sum \limits_{k=1}^d p_k r_k .$
Use the inequality $\displaystyle  ab\leq (\frac12 a+\frac12 b) ^2.$
From $ \displaystyle \frac12 a+\frac12 b=  \sum \limits_{k=1}^d p_k( \frac12 q_k+ \frac12 r_k). $
we get
\beq
\label{ineq1}
\left  [ \sum \limits_{k=1}^d p_k q_k \right ]\left  [ \sum \limits_{k=1}^d p_k r_k \right ] \leq
\left [ \sum \limits_{k=1}^d p_k( \frac12 q_k+ \frac12 r_k) \right ]^2.
\eeq
For the vectors $v$ and $v'$ in $\R^d$ with components $v_k= p_k$ and $v'_k= \frac12 q_k+ \frac12 r_k$
use Cauchy Schwartz inequality to bound their dot product,  $\langle v, v' \rangle $ by the product of their Euclidean norms $|v|\; |v'|. $
We get
\bdima
0\leq  \sum \limits_{k=1}^d p_k( \frac12 q_k+ \frac12 r_k) \leq
\sqrt{  \sum \limits_{k=1}^d p_k^2} \; \sqrt{\leq  \sum \limits_{k=1}^d ( \frac12 q_k+ \frac12 r_k)^2}.
\edima
or equivalently
\beq
\label{ineq2}
\left [ \sum \limits_{k=1}^d p_k( \frac12 q_k+ \frac12 r_k) \right ]^2 \leq
\left  [ \sum \limits_{k=1}^d p_k^2 \right ]\left  [ \sum \limits_{k=1}^d \left ( \frac12  q_k + \frac12 r_k \right ) ^2\right ].
\eeq
Combine (\ref{ineq1}) and (\ref{ineq2}) to reach the conclusion of Proposition \ref{FancyIneq}.
\end{proof}
Use  Proposition \ref{FancyIneq} with $d=n$ and non-negative scalars $p_k = x_k^2, \; q_k=c_k$ and $r_k = d_k)$.  Conclude that the net contribution of the last two terms of (\ref{almostWdot}) is negative and therefore

\bdima
\dot W \leq - K W_k^2.
\edima
The solution to the differential equation $\dot a= -K a^2$  is $\displaystyle a(t)= \frac{a(0)}{1+a(0)K t}.$
We get that
$\displaystyle W (t)  \leq \frac{W(0)}{1+W(0)Kt}.$ From $x_M(t) ^2 \approx W(t) $ we conclude
$$  x_M(t) \leq  \frac{ K''\sqrt{W(0)}}{\sqrt{1+t W(0)K}}.$$
This proves the asymptotic stability of the origin, and the stated decay rate.
\end{proof}

\vskip.2cm

Note that the functions $A(t) , \cos^2 t $ and $E(t)$ defined in (\ref{AE}), used  in setting up the polar coordinates system (\ref{polar}) are periodic with common period $\pi.$ For these coefficient functions  the averages mentioned in Proposition \ref{StableAlmost} can be calculated as:
\beq
\label{means}
\begin{array}{l}
 \mu (A)=-\frac{59}{200}= - 0. 295, \\
 \mu( \frac14(\cos^2 t +E(t))^2)= \frac{437}{1600} =.273125,\\
 \mu (A)+ \mu( \frac14(\cos^2 t +E(t))^2)= -\frac{7}{320}.
 \end{array}
 \eeq

\section{Proving the Stability of the Translating States and of the Convergence Rates }\label{Section_Stability}

In this section we complete the proof of stability of the translating state of (\ref{Main}) by working with (\ref{polar}), the polar coordinates reformulation  of the dynamics on the subspace $H\times H$ of the central manifold system (\ref{ODEwy}).  We also prove that the direction of motion for the center of mass of the swarm (\ref{Main}) has a convergent angle $\Theta (t)$ and that all agents' velocities  align with the mean velocity $V.$  We give the rates of convergence to the limit configuration for all the relevant physical coordinates of (\ref{Main}).

\subsection{Proving the Stability of the Origin}\label{SubSection_Stability}

This subsection re-purposes the arguments of Subsection \ref{Subsection_TimeDependent} to prove that the origin of (\ref{polar}) is asymptotically stable, with convergence rate of $\frac{1}{\sqrt t},$ as stated in Theorem \ref{general}.

Due to agents having a common oscillation frequency, one anticipates that the phase angles $\theta_k$ will synchronize, that
$\dot a_k$ will become interchangeable with $da_k/ d\theta_k$ and the coefficient functions  $A(\theta_k)$ interchangeable with the time-dependent (rather than phase-dependent) trigonometric polynomials $A_k(t)$. However, although the differences $\dot \theta _{k}-1$ are small for some agents  $k, $  that may not hold true for the particles that are much closer to the origin than those further away.  For particles closest to the origin it is possible for
$\dot{\theta}_ k -1 \approx\frac{m}{a_k}\mco_0 $
 to be very large even though $m=\mco_3,$ meaning that the phase of such particles can be  wildly out of sync from the particles with larger amplitudes. Figure \ref{AverageDecay} illustrates that their amplitudes can have peculiar behaviour as well: while amplitudes of agents generally decrease over time, agents with much smaller magnitudes may undergo prolonged periods of amplitude growth.

\begin{figure}[h!]
\centering
\includegraphics[width=.9\textwidth]{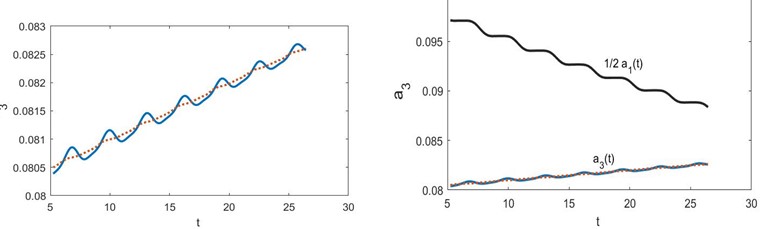}
\caption{
Left: the amplitude $a_3(t)$ of the agent closest to the origin; the dotted curve represents the average of $a_3$ over a window of length $2\pi.$  In average its oscillations increase to catch up with  the larger amplitudes of the swarm.
Right: a re-scaled illustration of the smallest amplitude $a_3$ and largest amplitude $a_1(t)$. In average the larger oscillations decrease.}
\label{AverageDecay}
\end{figure}

\begin{theorem}
\label{general}
Let $A$ and $E$ be the periodic functions introduced in  (\ref{AE}).
For any $\epsilon >0$ there exists $\delta>0$ such that any solution $(a_k(t), \t_k(t))$ of
\bdima
\begin{array}{l}
\dot a_k = -m\cos \t_k +a_k^3 A(\t_k) + a_k \cos^2 \t_k \frac{1}{n} \sum \limits_{l=1}^n a_l^2 E(\t_l)+\mco_4\\
 \\
\dot{\theta}_ k = 1+\frac{m}{a_k}\mco_0 + \mco_2, \; \; \; \; \mbox{ except at } \; \mathcal E{oo}
\end{array}
\edima
with $ \displaystyle \sum \limits _{k=1}^n a_k \cos \t_k=0$ and $m=\mco_3$ satisfies:
if $a_k(0)<\delta$ for $k=1..n$, then $a_k(t)<\epsilon$ for all $t>0$ and $k=1..n.$

Moreover, there exist $c_1, c_2$ that $ \displaystyle c_1 / \sqrt t \leq \frac{\max_k a_k(t) }{ \max_k a_k(0)} \leq c_2 )/\sqrt t $ for large $ t$.
\end{theorem}

\begin{proof}

We consider orbits that start with amplitude vector $a=(a_1, \dots a_n)$ in a small neighborhood  of the fixed point $\zn \in \R^{n}$.
We will use a  Lyapunov function akin to
the one in Section \ref{Subsection_TimeDependent},
modified to segregate the agents that oscillate much closer to the origin from the rest.

Per (\ref{means}) the functions $\displaystyle A(\t)+\frac{59}{200}$ and $\displaystyle \left (\frac12 \cos^2 \t +\frac12 E(\t)\right ) ^2 -\frac{437}{1600}$ have  periodic, thus bounded,   antiderivatives,
with period $\pi. $ Let  $C(\t)$ be an antiderivative of the latter, and $B(\t) $ be an  antiderivative of $\displaystyle A(\t) +\frac{59}{200}$ , both with strictly positive range. \footnote{ The actual formulas for the antiderivatives $C(\t)$ and $B(\t)$  are unimportant in our calculations; however, one could use:
 $\displaystyle
B(\t )= \int A(\t)  +\frac{59}{200}d\t=
 -\frac{17}{100} \sin (2\t)-\frac{9}{800} \sin (4 \t)-\frac{3}{25} \cos ^4(\t)+ 1 $ and
$\displaystyle
C(\t) = 1+\frac{231 }{40000}\sin (4 \t)-\frac{1}{400} \cos (2 \t)-\frac{43 }{160000}\cos (4 \t).$}

 We use the following influence-diminishing function to  weigh the impact of agents.
Let $T_{norm} :\R \to [0,1]$ be the Lipschitz continuous function given by $T_{norm} (r) = 1$ for $|r|\geq 1$ and $T_{norm} (r) = r^2$ for $r$ in $[-1, 1].$ Its graph is illustrated in Figure \ref{TFigureT}.
\begin{figure}[h!]
\centering
\includegraphics[width=.5\textwidth]{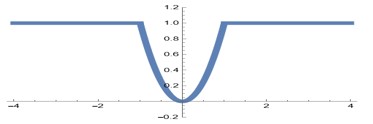}
\caption{The function $T_{norm}$.}
\label{TFigureT}
\end{figure}

Given a point $a=(a_1, \dots a_n)$ in  $\R^{n}$ denote by $T_k(a)$ or simply by $T_k$ the value of
$\displaystyle T\left (\frac{a_k}{a_M^{2.75}} \right ). $ Recall that $a_M$ denotes the largest magnitude. The power $2.75$ was selected to be just below the order of $m$;  $m =\mco_3. $  We have:
\beq
T_k(a) = \begin{cases}
\frac{a_k^2}{a_M^{5.5}} & \text { if  } a_k< a_M^{2.75} \\
1 & \text{ if  } a_k\geq a_M^{2.75}
\end{cases}
\eeq
For notational convenience let $L=L(t)$ and $S=S(t) $ denote subsets of $\{1, \dots n \} $ that correspond to agents whose magnitude at that instant is relatively large ($L$) or relatively small ($S$):
$$ L=\{ k| a_k(t) \geq a_M^{2.75}(t) \}, \; \;  S=\{ k| a_k(t) < a_M^{2.75}(t) \}. $$
Note that if at some time $k\in L$ then $T_k =1, $ and if $k\in S$ then $T_k=\frac{a_k^2}{a_M^{5.5}} .$

\begin{lemma}
\label{smalldT}
 If $ a=(a_1, \dots a_n)$  is a solution to
the system in Theorem \ref{general}, i.e. system (\ref{polar})
, then
$$ \frac{d}{dt} T_k(a(t)) = \mco _{0.25} \; \; \mbox{ for all } k=1..n$$
\end{lemma}

\begin{proof}
Fix $k$ in $1..n.$ The continuous function $T_k$ is almost everywhere differentiable and has a weak derivative given almost everywhere by
$$ \dot T_k= \begin{cases}
\frac{d}{dt} \frac{a_k^2}{a_M^{5.5}} & \text { if  } a_w< a_M^{2.75} \\
0  & \text{ if  } a_w\geq a_M^{2.75}
\end{cases}
$$
If $k \in L$ then $\dot T_k=0.$ For  $k \in S$ use
\bdima
\frac{dT_k}{dt}= 2 \frac{a_k \dot a_k }{a_M^{5.5}}- 5.5 \frac{a_k^2\dot a_M}{a_M^{6.5}}.
\edima
Note that for al indexes $k$ the lowest terms in $\dot a_k$ are of order three, so
$$ \frac{a_k |\dot a_k | }{a_M^{5.5}}=\frac{a_k}{a_M^{2.75}}\frac{|\dot a_k|}{a_M^3}a_M^{0.25}\leq \mco_{0.25}\; \; \mbox{ since } k\in S.$$
For the second term of $\dot T_k$ use $\dot a_M =\mco_3 $ and
$\displaystyle \frac{a_k^2}{a_M^{3.5}}= \left( \frac{a_k}{a_M^{2.75}} \right )^2 a_M^2$ to conclude that for $k \in S$ we have  $\displaystyle \frac{a_k^2\dot a_M}{a_M^{6.5}} = \mco _2.$ Overall the sum of the two terms is the larger of the parts, i.e. $\mco_{0.25}.$
\end{proof}

\begin{lemma}
\label{smalldTheta}  If $ a_1, \dots a_n$ and $\t_1, \dots \t_n$  are solutions to
the system in Theorem \ref{general}, i.e. system (\ref{polar}), then
\bdima
\left ( \frac{d\t_k(t)}{dt} -1 \right )T_k (a) = \mco _{0.25} \; \; \; \mbox{ for all } \; k=1..n.
\edima
\end{lemma}

\begin{proof}
The functions $T_k$ are between zero and one, and $\displaystyle \dot \t_k -1 = \frac{m}{a_k}\mco_0+\mco_2$ so we only need to prove that $\displaystyle \frac{m}{a_k} T_k=\mco_{0.25}.$

If $k$ is an index from $L$ then $T_k=1$ and $\displaystyle \frac{m}{a_k}=\frac{m}{a_M^3} \frac{a_M^{2.75}}{a_k}\; a_M^{0.25}.$ The first fraction of the product is bounded, the second is less or equal to one since $k$ is an index in $L$, leaving  $\displaystyle \frac{m}{a_k}=\mco_{0.25}.$

If $k$ is an index from $S,$ then $\displaystyle T_k= \frac{a_k^2}{a_M^{5.5}} $ and thus
\bdima
\frac{m}{a_k} T_k=\frac{m}{a_k}\frac{a_k^2}{a_M^{5.5}}=\frac{m}{a_M^3}  \frac{a_k}{a_M^{2.75}}\; a_M^{0.25} = \mco _{0.25}.
\edima
We used the fact that $\frac{m}{a_M^3} = \mco_0$ and that when $k\in S$ the ratio $\frac{a_k}{a_M^{2.75}}$ is under one.
\end{proof}

\begin{lemma}
\label{small1}  $\displaystyle a_k^2 (1-T_k(a)) \leq a_M^{5.5}$ and $\displaystyle a_k^2 (1-\sqrt{T_k(a)}) \leq a_M^{5.5}$for all $k.$
\end{lemma}

\begin{proof} If $k$ is in $L$ then $T_k(a) =1$  and the inequalities hold.
If $k$ is in $S,$ then $a_k ^2\leq a_M^{5.5}$ and  both $  1-T_k$ and $1-\sqrt{T_k}$ are in $[0,1]$.

\end{proof}

\noindent
Returning to the proof of Theorem \ref{general} :

Define the function $W_k=W_k(a_1(t), \dots \t_n(t))$  as
\beq
\label{WT}
W_k(t)= \frac{a_k^2(t)}{2a_k^2(t)\left ( B(\t_k(t))+ \frac{1}{n} \sum \limits_{l=1}^{n} T_l (a) C(\t_l) \right )+1}
\eeq
where $B$ and $C$ are the positive, bounded  antiderivatives of $59/200+A, $ and of
\newline  $ \displaystyle (\frac12 \cos^2 \t + \frac12 E(\t) ) ^2 -437/1600.$

Define $W$ as
$\displaystyle  W= \frac{1}{n} \sum \limits_{l=1}^{n} W_l.$
Note that along trajectories other than the origin,  the functions $W_k, W$ are continuous functions of time.
While $a$ is in small neighborhood of the origin, $W_k$ are the composition between the Lipschitz continuous and piece-wise $C^5$ functions $1/(1+x)$, $ a_k^2(t), B(\t_k(t)), C(\t_k (t)) $ and
 $T_l,$ therefore the functions $t\to W_k(t)$ are  differentiable almost everywhere,  with weak derivatives that can be computed with the usual chain rule for almost all $t$.
The denominators of $W_k$  are very close to $1,$ so each $W_k$ is comparable in value to $a_k^2$. In fact
$W$ and $a_M^2$ are only within a factor of $n$ from each other.

We claim that $W$ satisfies the differential inequality  $\displaystyle \dot  W \leq -KW^2 $ for some positive constant $K.$ We start by estimating just $\dot W_k.$

From Remark \ref{quotientrule} we get that
\bdima
\dot W_k= \left [2 a_k \dot a_k - 2 a_k ^4 \left (  \frac{dB}{d\t}\rvert_{\t_k} \dot{\theta_k} +\frac{1}{n} \sum \limits_{l=1}^{n} \frac{dT_l}{dt}C(\t_l)+
\frac{1}{n} \sum \limits_{l=1}^{n} T_l \frac{dC}{d\t}\rvert_{\t_l}\dot{\theta_l}
 \right )   \right ] (1+\mco_2 )
\edima
Recall that $a_k \dot a_k$ and  $a_k^4 \dot{ \theta_k}$ are $\mco_4$ and
that $ \frac{dB}{d\t},  \frac{dC}{d\t}, \frac{dT_l}{dt}$ and $T_l$ are bounded, implying that the above bracket is $\mco_4.$ We get that
\beq
\label{quotientW}
\frac{1}{2}\dot W_k= \left [ a_k \dot a_k - a_k ^4 \left (  \frac{dB}{d\t}\rvert_{\t_k} \dot{\theta_k} +\frac{1}{n} \sum \limits_{l=1}^{n} \frac{dT_l}{dt}C(\t_l)+
\frac{1}{n} \sum \limits_{l=1}^{n} T_l \frac{dC}{d\t}\rvert_{\t_l}\dot{\theta_l}
 \right )   \right ] +\mco_6.
\eeq

Estimate the four terms of $\frac12 \dot W_k.$

\vskip2mm
\noindent
For the first term, use (\ref{polar}) to rewrite $ a_k \dot a_k $ within an $\mco_5$ approximation:
\bdima
a_k\dot a_k=- ma_k\cos \t_k + a_k^4A(\t_k) + a_k^2 \cos^2 \t_k \frac{1}{n} \sum \limits_{l=1}^{n} a_l^2 E(\t_l) +a_k \mco _4
\edima
For third term, $ a_k^4 \frac{1}{n} \sum \limits_{l=1}^{n} \frac{dT_l}{dt}C(\t_l),$ use Lemma \ref{smalldT}  to conclude that it is $\mco_{4.25} .$

\vskip2mm
\noindent
For the second and fourth terms use (\ref{polar}) and Lemma \ref{smalldTheta}, respectively:
\bdima
\begin{array}{l}
a_k^4  \frac{dB}{d\t}\rvert_{\t_k} (\dot{\theta_k}-1) = \frac{dB}{d\t}\rvert_{\t_k} (m\; a_k^3\mco_0+ a_k^4 \mco_2 ) =\mco_6 \; \; \mbox{ and}\\
\\
a_k^4 T_l \frac{dC}{d\t}\rvert_{\t_l} (\dot{\theta_l}-1) =a_k^4\frac{dC}{d\t}\rvert_{\t_k} \mco_{0.25}= \mco_{4.25}
\end{array}
\edima
Thus in (\ref{quotientW}), replacing $\dot \t_k, \; \dot \t_l$ by $1$ and $\dot T$ by $0$ keeps the approximation within $\mco_{4.25}$:
\beq
\label{Wdot25}
\begin{array}{l}
\frac12 \dot W_k=
\left [- ma_k\cos \t_k + a_k^4A(\t_k) + a_k^2 \cos^2 \t_k \frac{1}{n} \sum \limits_{l=1}^{n} a_l^2 E(\t_l)\right ]+ \\
 a_k ^4 \left ( - \frac{dB}{d\t}\rvert_{\t_k}  -
\frac{1}{n} \sum \limits_{l=1}^{n} T_l \frac{dC}{d\t}\rvert_{\t_l}
 \right )
+\mco_{4.25}
\end{array}
\eeq
From Lemma \ref{small1} we get  $\displaystyle  a_k^2 \cos^2 \t_k \frac{1}{n} \sum \limits_{l=1}^{n} a_l^2 E(\t_l) $ and
$\displaystyle   a_k^2 \sqrt{T_k} \cos^2 \t_k \frac{1}{n} \sum \limits_{l=1}^{n} a_l^2 \sqrt{T_l} E(\t_l)$ are within $\mco_{5.5}$ from each other.

Substitute $ \frac{dB}{d\t}= A+\frac{59}{200}$ and $-\frac{dC}{d\t}= \frac{437}{1600}- \left( \frac12 \cos^2 \t + \frac12 E(\t)\right) ^2$ into (\ref{Wdot25});   the terms $a_k^4A(\t_ k)$ cancel and we get:

\beq
\label{Wdot26}
\begin{array}{l}
\frac12 \dot W_k=
\left [- ma_k\cos \t_k -\frac{59}{200} a_k^4 + a_k^2 \sqrt{T_k}\cos^2 \t_k \frac{1}{n} \sum \limits_{l=1}^{n} a_l^2 \sqrt{T_l}E(\t_l)\right ]+ \\
\frac{437}{1600} a_k ^4
\frac{1}{n} \sum \limits_{l=1}^{n} T_l - a_k ^4 \left (
\frac{1}{n} \sum \limits_{l=1}^{n} T_l \left( \frac12 \cos^2 \t _l + \frac12 E(\t _l)\right) ^2 \right )
+\mco_{4.25}
\end{array}
\eeq
Since $0\leq T_l\leq 1$,  $  \frac{1}{n} \sum \limits_{l=1}^{n} T_l \leq 1$
and $ \displaystyle -\frac{59}{200}+ \frac{437}{1600} \frac{1}{n} \sum \limits_{l=1}^{n} T_l \leq -\frac{7}{320}.$
Add the equations (\ref{Wdot26}) corresponding to $k=1..n,$ and divide by $n.$ Use the fact that $\sum \limits_{l=1}^{n}a_k \cos\t_k=0$.

\beq
\label{Wdot27}
\begin{array}{l}
\frac12 \dot W\leq
-\frac{7}{320}\frac{1}{n} \sum \limits_{k=1}^{n} a_k^4 + \left [ \frac{1}{n} \sum \limits_{k=1}^{n}a_k^2 \sqrt{T_k}\cos^2 \t_k \right ]
\left [ \frac{1}{n} \sum \limits_{l=1}^{n} a_l^2\sqrt{T_l} E(\t_l)\right ]+ \\
 -\left [ \frac{1}{n} \sum \limits_{k=1}^{n} a_k ^4 \right ]\left (
\frac{1}{n} \sum \limits_{l=1}^{n}\left (  \frac12 \sqrt{T_l} \cos^2 \t_l + \frac12 \sqrt{T_l}E(\t_l)\right) ^2 \right )
+\mco_{4.25}
\end{array}
\eeq
Use Proposition \ref{FancyIneq} with  $d=n$  and $p_k=a_k^2,\;  q_k= \sqrt{T_k}\cos^2 \t_k $ and $ r_k=\sqrt{T_k}E(\t_k) $
to conclude that the net contribution of the two products of sums in (\ref{Wdot27}) is negative, therefore
\bdima
\dot W \leq -\frac{7}{160} \frac{1}{n} \sum \limits_{k=1}^{n} a_k^4+ +\mco_{4.25}.
\edima
The function $W$ and $a_M^2$ are comparable in scale, therefore there exists a constant $K>0$ that depends only on $n$ such that
$$ \dot W \leq -KW^2,$$
meaning that for all $ t>0$
\bdima
W(t) \leq \frac{W(0)}{1+W(0)Kt} \; \; \mbox{ and } \frac{\max_k a_k(t) }{ \max_k a_k(0)} \leq \frac{c_2}{\sqrt{1+W(0)Kt}}.
\edima

To get the lower bound of $C/\sqrt t$ for $a_M(t),$ or equivalently the lower bound of $1/t$ for $W(t)$, return to (\ref{Wdot26}).
Let $K_1$ the maximum value of the periodic function $ \left (\frac12 \cos^2 \t + \frac12 E(\t) \right ) ^2.$ We have
$$\dot W_k \geq -m a_k \cos \t_k -\frac{59}{200} a_k^4 +0-K_1 a_k^4 + \mco_{4.25}$$
and on account that $\sum _{k=1}^n m a_k \cos \t_k=0$ conclude
$$\dot W \geq  -(\frac{59}{200}+K_1) a_k^4 + \mco_{4.25} \geq -C W^2$$
ensuring that $W$ has a lower bound comparable to  $1/t.$
\end{proof}


\begin{corollary}
\label{RateCor}
 For initial conditions $Z_0 \in H\times H \times \R^{2n-1}$ near the fixed point $Z_f $, $Z_0$ not on the stable manifold $W^s(Z_f),$ there exists $C>0$ depending on $Z_0$ such that
 \beq
 \label{EqRateCor}
 \max_k \left (|r_k(t)-R(t)|+ |\dot r_k(t) -V(t)| \right ) \geq \frac{C}{\sqrt t} \; \; \mbox{ for large  } t.
 \eeq
\end{corollary}

\begin{proof} Let $Z_0 \in H\times H \times \R^{2n-1}$ be near the fixed point $Z_f, $ with $Z_0 \notin W^s(Z_f).$
Theorem \ref{general} implies that the flow  of (\ref{State}) within $\in H\times H \times \R^{2n-1}$ is asymptotically stable near $Z_f,$ therefore there exists an initial condition $Z_{0,c}$ in the central manifold, such that the two solutions $\psi_t(Z_0)$ and $\psi_t(Z_{0,c})$ are exponentially close for $t$ in $[0, \infty).$  Let $r_{c,k}, R_c, V_c, w_{c, k}, \dots u_{c, k}$ denote the position vectors, velocities and their components for the orbit of $Z_{0,c}.$
Since $Z_0 \notin W^S(Z_f),$ the projection onto the $ (w,y)$ subspace of the point $ Z_{0,c}$ is not the origin, so $a_{c,M}(0)= \max_k \sqrt{w_{c,k}^2(0)+y_{c,k}^2(0)} \neq 0.$

For the solution originating at $Z_0$ we have
\bdima
|r_k(t)-R(t)|^2+ |\dot r_k -V|^2 = x_k^2+y_k^2+(u_k-|V|)^2 +w_k^2\geq w_k^2+y_k^2=a_k^2,
\edima
thus $\left ( |r_k(t)-R(t)|+ |\dot r_k -V| \right )\geq  a_k.$
Due to the exponential proximity of the two orbits, there exist positive $C$ and $\eta $ such that
\bdima
|a_k(t) -a_{c, k}(t) | \leq Ce^{-\eta t} \; \mbox{ for } t \in [0, \infty).
\edima
We get
\bdima
\max_k \left (|r_k(t)-R(t)|+ |\dot r_k -V| \right ) \geq \max_k a_k(t)  \geq \max_k a_{c,k}(t)-Ce^{-\eta t},
\edima
with $\max_k a_{c,k}(t) \geq \frac{a_{c,M}(0)}{\sqrt t}, $ per Theorem \ref{general}.
This  proves (\ref{EqRateCor}).
\end{proof}

Theorem \ref{general} implies that the flow on the central manifold of (\ref{State0}) is asymptotically stable in $H\times H$.  It also implies  the asymptotic stability of the origin for the system (\ref{State0})  with decay rate of order $1/\sqrt t $,
and the asymptotic stability of $Z_f=col( \zn, \zn, \znn, \onen )$ for (\ref{State}) in $H\times H \times \R^{2n-1}.$
 To complete the proof of Theorem \ref{MainTh}
 we are left to show that the direction of $V$ is convergent and close to that of $V(0),$ for which it is sufficient  to prove that $\int _0^\infty |m|dt$ converges,  per
Remark (\ref{recoverflow}).
  The solutions of the system (\ref{State}) stemming from initial conditions near $Z_f$  evolve exponentially close to a solution on the center manifold. It is therefore enough to limit the investigation of $\int _0^\infty |m|\; dt$ to trajectories from the central manifold.

\subsection{Limit Behavior of the Mean Velocity V and of the Physical Coordinates of the Swarm}

In this subsection we recast the results obtained for the amplitudes $a_k$ in terms of the physical coordinates of the swarm (\ref{Main}): the positions  $r_k$ and $R,$ and the velocities $\dot r_k $ and  $V$.
We assume that the initial conditions for (\ref{State}) are from a small neighborhood $\mathcal B$ of $Z_f$, within its basin of attraction.

We begin by addressing the limit behavior of the direction of motion, i.e.
the polar angle $\Theta $ satisfying $\displaystyle V(t)=|V(t)| e^{i\Theta (t)}.$
Recall that the change in $\Theta (t) $  is given by
$\displaystyle  \frac{d\Theta}{dt}=m , $ according to (\ref{mdefine}) and (\ref{mandS2}).

\begin{remark}
\label{remarkVinfty} There exists $\Theta _\infty $  depending on $Z_0$ and $c>0$ such that
\bdima | \Theta (t) - \Theta _\infty| \leq \frac{c}{\sqrt t}  \; \; \mbox{ for large } t. \edima
Moreover, $\lim _{t \to \infty} V(t) = ( \cos \Theta _\infty, \sin \Theta _\infty). $
\end{remark}
Proof: Due to the exponentially close proximity of arbitrary orbits to orbits from the central manifold, it is enough to prove the stated assertion for the central manifold flow.

Recall that $\Theta (t) =\Theta _0+\int _0^t m(s) ds.$ For $Z_0$ on the central manifold $m(t)=\mco_3.$  We get that for large $t$ the absolute value of $m$ is below a multiple of $a_M(0) ^3 t^{-3/2}$ which has a convergent integral at infinity.
This proves the existence of $\displaystyle \lim_{t \to \infty} \Theta (t),  $ denoted $ \Theta _\infty.$ Note that $| \Theta _\infty-\Theta (t)|$ is bounded above by $\displaystyle \int_t ^\infty \frac{c a_M(0) ^3}{s^{3/2}} ds$ which is of order $a_M(0) ^3 t^{-1/2}.$

Let $\displaystyle V_\infty= (\cos \Theta _\infty, \sin \Theta _\infty).$
From $V(t) =|V(t)| (\cos \Theta (t), \sin \Theta (t)) $ with $|V|\to 1$ and $ \Theta (t) \to \Theta _\infty $ we conclude $V(t)$ converges to $(\cos \Theta _\infty, \sin \Theta _\infty).$

\begin{remark} If $\psi_t(Z_0) $ is a solution of (\ref{State}) not in $W^s$ then there exists $c>0$ such that the mean speed satisfies $1-|V|\geq \frac{c}{t}$ for large $t.$ Moreover, there exist initial conditions $Z_0$ for which the center of mass $R(t)$ does \emph{not} stay within bounded distance from rectilinear motion with velocity $V_\infty, $
in the sense that for any $T >0$
\bdima
| R(t+T)- R(T)- t V_\infty  | \to \infty \; \mbox{ when  } t \to \infty.
\edima
The center of mass falls behind exceedingly large distances compared to where rectilinear motion would have located it.
\end{remark}

\begin{proof} If $\psi_t(Z_0) $ is not in $W^s$, then it is exponentially close to a nontrivial solution from the central manifold, $\psi_t(Z_0^c).$ It suffices to prove the bound for  $Z_0^c.$
Note that for $Z_0^c\neq 0 $ we have
\bdima
1-|V|=-\overline z=
\frac{1}{n} \sum_{l=1}^n  (\frac38 w_l^2+ \frac28 w_l y_l + \frac18 y_l^2 )\geq \frac{2-\sqrt 2}{8}  \frac{1}{n} \sum_{l=1}^n a_k^2\geq \frac{c}{t}.
\edima
To prove the existence of configurations with large gap between the center of mass $R$ and the rectilinear motion, consider agents with mirror symmetry with respect to $x-$axis (i.e. $m=0$). Due to the unidirectional motion, we have that $V(t)= (|V|, 0)$ with $V_\infty =(1,0)$ and we get that $R$ moves along the $x-$ axis: $R(t)=(|R(t)|, 0).$
Identifying $R$ with its $x-$ component we get
\bdima
t V_\infty -( R(t+T)- R(T))  = \int _T^{t+T} (1-|V(s)| ) ds\geq  \int _T^{T+t} \frac{c}{s} ds\geq c\ln \frac{t+T}{T}\to \infty
\edima
as $t\to \infty$.
\end{proof}

\begin{remark}
\label{driftremark}
For $n\geq 3$ the estimates of $\mco(t^{-3/2})$ for $m$ and for the drift $\Theta$ are sharp: there exist initial conditions $Z_0$ for which the solution $\psi_t(Z_0) $  of (\ref{State}) satisfies: there exist $c>0$ and times $t', t''$ going to infinity such that $m(t') \geq c/(t')^{3/2}$ and such that the variation in the directions of motion $\Theta (t)$  satisfies
$$ |\Theta (t'')-\Theta (t')| \geq C/(t')^{3/2} $$
\end{remark}

\begin{proof} In order to prove that the estimate $m=\mco_3$ is sharp, consider a configuration  with $n-1$ of the particles having identical initial conditions,  thus identical motion for all $t$: $w_k=w(t) , y_k=y(t) $ for $k=1..n-1$. Necessarily, the last particle has $w_n =-(n-1) w, \; y_n=-(n-1) y.$ Let $ a_1$ and $ \t$ denote the common amplitude and polar angle for the first $n-1$ particles.

From Theorem \ref{general} there exists $c_3>0$ such that $a_1(t) \geq c_3/\sqrt t.$ For this configuration
on account that $\sum w_k=0$  we get
\beq
\label{largem}
\begin {array} {l}
m= -\frac{1}{n}\sum \limits_{k=1}^n \left( \frac{17}{25}w_k^2-\frac{12}{25}w_k y_k+ \frac{8}{25}y_k^2 \right )w_k +\mco_5 =\\
(n-1) (n-2) \left( \frac{17}{25}w^2-\frac{12}{25}w y+ \frac{8}{25}y^2 \right )w +\mco_5\geq\\
c \frac{1}{t \sqrt t}  \left( \frac{17}{25}\cos ^2 \t-\frac{12}{25} \cos \t \sin \t  + \frac{8}{25}\sin^2\t  \right )\cos \t +\mco_5
\end{array}
\eeq
whenever the trigonometric function is positive.
\begin{figure}[h!]
  \centering
  \includegraphics[width=0.4\textwidth]{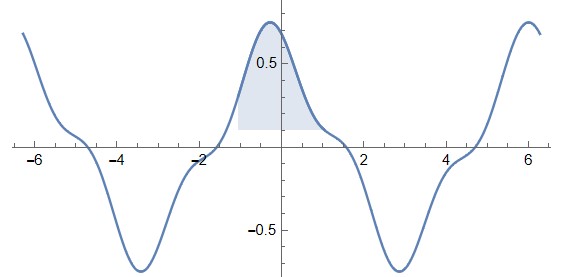}
  \caption{The graph of  $y=\left( \frac{17}{25}\cos ^2 \t-\frac{12}{25} \cos \t \sin \t  + \frac{8}{25}\sin^2\t  \right )\cos \t $, shown on $[-2\pi, 2\pi]; $ the shaded part has $\t \in [-\pi/3, \pi/3]$ and $ y>0.10.$
   }
  \label{sizeofdrift}
\end{figure}

Recall that $\frac{d\t}{dt}= 1+\mco_2 = 1+\mco (1/t),$ meaning that the polar angle $\t (t)$ is strictly increasing  (to infinity), is Lipschitz (with constant close to one for intervals near infinity), and  $\t(t)$ is comparable in size with $t.$
Let $N$ be large. Let $t'=t'(N)$ and $t''=t''(N)$ be time values when $\t(t')=2N\pi-\pi/3$ and $\t(t'')=2N\pi+\pi/3.$ We have that both $t'$ and $t''$ are comparable to $N$, and that for $t$ in $[t', t'']$ the amplitude $a_1(t)$ is comparable to $1/\sqrt N$. For such times, modulo $2\pi$ the polar angle $\t$
reduces to an angle in $[-\pi/3, \pi/3] $ where the trigonometric function
$\left( \frac{17}{25}\cos ^2 \t-\frac{12}{25} \cos \t \sin \t  + \frac{8}{25}\sin^2\t  \right )\cos \t $ is above $0.10$, see Figure \ref{sizeofdrift}.
These estimates,  and (\ref{largem})  prove
$$m(t) \geq  \frac{0.10 c}{t^{3/2}}\; \mbox{  for all } t \; \mbox{in the interval } [t', t''].$$
For this configuration, the drift of the direction of motion,  $ \displaystyle \Theta(t'')-\Theta(t') $  satisfies:
$\displaystyle
\Theta(t'')-\Theta(t')= \int _{t'}^{t''} m(t) dt \geq \int _{t'}^{t''} \frac{0.10 c}{t\sqrt t} dt .$
Next we show that the length of the interval $[t', t'']$ is bounded below and above by absolute constants:  the functions of $t \to \t (t)$ and its inverse $ \t \to t $  have Lipschitz constants close to one, so there exist positive constants $c_4$ and $c_5$  such that
$$ \frac{2\pi}{3}=|\t(t'') -\t(t') |  \leq c_4 |t''-t'| \leq c_4 c_5 |\t(t'') -\t(t') |= c_4c_5 \frac{2\pi}{3} .$$
We conclude that $\displaystyle \int _{t'}^{t''} \frac{0.10 c}{t\sqrt t} dt $ is comparable to $(t')^{-3/2}$ and that the drift $\Theta(t'')-\Theta(t') $ exceeds $(t')^{-3/2}.$
\end{proof}

\section{Appendix}

{\bf Proposition \ref{NoSymmetries} }
If among the $n$ agents of the system (\ref{Main}) there exist $p \geq 3 $  agents (assumed to be the first $p$) whose motion is symmetrically distributed about the center $R$ of the swarm, i.e.
\beq
\label{eqSymmetries}
 r_k(t)= R(t)+e^{2\pi i \frac{(k-1)}{p}}( r_1(t)-R(t)) \; \mbox{ for }\; k=1 ..p,
\eeq
then either $r_1(t)=\dots = r_p(t)$ for all $t\in \R$, or $R(t)=R(0)$ for all $t\in \R.$

\begin{proof}

Within this proof, we identify all the planar  vectors as points in $\mathbb C$ so that we can multiply, divide, and take complex conjugate.

Assume that there exists $p\geq 3$ such that (\ref{eqSymmetries}) holds while the center of mass is not stationary, i.e. there exists a time interval $ I_1=(t_1, t_2)$  with  $ V(t) \neq 0$ for $ t \in I_1.$  We want to show that there exists an open time interval such that all $p$ particles' locations equal $R(t).$ (If such an interval exists, it means the velocities for the $p$ particles match $V(t)$ in that interval, so the $p$ particles have identical initial conditions; by the uniqueness of solutions to (\ref{Main}) we get that the particles satisfy $ r_1(t)=\dots = r_p(t)$ for all $t\in \R$.)

To simplify notations, {\em within this proof} \footnote{Outside of this Appendix, we used $\t_k$ and $m$ to denote real-valued functions $\t_k=\t_k(t)$ and $m(t)$ with a different meaning. In this proof $\t_k$ are constants, and $m(t)$ is complex-valued.}
let $\theta_k = 2 \pi  \frac{(k-1)}{p}$ for $ k=1.. p$ and let $\sigma (t), m(t)$ and $m_k(t)$ be complex valued function defined by:
\beq
\label{defineAppsigma}
 \sigma (t)=  r_1(t)- R(t), \;  \; m(t) = \frac{\dot \sigma }{V}, \; \mbox{ and } m_k(t)=e^{-i \theta_k}+ m(t)
 \eeq
We get that for $k=1.. p,$
\beq
\label{Apptk}
 r_k = R+ e^{i \theta_k} \sigma  \; \mbox{ and } \; \dot r_k= V+e^{i \theta_k}Vm .
 \eeq
To prove the appendix, we need to show that if $p\geq 3,$ then there exists a time interval $I_2$  where $\sigma (t)=0$ for $t\in I_2.$

Using (\ref{Apptk}), the first $p$ equations of (\ref{MainC}) become
$$
\dot V+ e^{i \theta_k} \dot V m + e^{i \theta_k}V \dot m =
(1-|V|^2 |1+e^{i \theta_k}m|^2)V(1+e^{i \theta_k}m ) - e^{i \theta_k} \sigma $$
\bdima
\dot V e^{i \theta_k}( e^{-i \theta_k}+ m(t)) + e^{i \theta_k}V \dot m=
(1-|V|^2 |e^{-i \theta_k}+ m|^2)Ve^{i \theta_k}(e^{-i \theta_k}+m ) - e^{i \theta_k} \sigma
\edima
Divide by $e^{i \theta_k}$; recall that  $e^{-i \theta_k}+ m(t)=m_k.$ We get
$$ \dot V m_k +V \dot m = (1-|V|^2|m_k|^2)Vm_k -\sigma. $$
Divide by $V|V|^2$ and collect the $m_k$ terms to get that for $k=1..p$ and $t\in I_1$
\beq
\label{Absurd1}
\frac{V\dot m+ \sigma}{V|V|^2} = m_k\left ( -\frac{\dot V}{V|V|^2}+\frac{1}{|V|^2} -|m_k|^2 \right )
\eeq
 Employ the notations:  $\alpha (t)=\frac{1}{|V|^2} -\frac{\dot V}{V|V|^2}$ and $\beta(t)=\frac{V\dot m+ \sigma}{V|V|^2}$ to further simplify (\ref{Absurd1}) to:
 \beq
 \label{alphabeta}
 \beta(t)=  m_k(t)\left (\alpha (t) -|m_k(t)|^2 \right), \; \mbox{ for } k=1, \dots p.
 \eeq

\begin{lemma} If $\alpha, \beta, m \in \mathbb C  $ are such that
\beq
 \beta= (m+e^{-i \theta_k}) \left (\alpha  -|m+e^{-i \theta_k}|^2 \right) \; \mbox{ for } \t_k=\frac{2\pi (k-1)}{p}, k=1, \dots p,
  \label{alphabeta2}
 \eeq
then $\alpha $ is the positive real number $\alpha= 2|m|^2+1$ and $\beta= m (|m|^2-1).$
\label{AppLemma}
\end{lemma}

\begin{proof}  Note that
\beq
\label{mkconjugate}
  |m+e^{-i \theta_k}|^2 = |m|^2+1+ m e^{i\t_k}+\overline{m}   e^{-i\t_k}
\eeq
and that the $p$ roots of unity, $\displaystyle e^{i\t_k}, $ satisfy:
\beq
 \sum_{k=1}^p e^{i\t_k}=0, \;  \sum_{k=1}^p e^{-i\t_k}=0, \;  \sum_{k=1}^p e^{i 2\t_k}=0\; \mbox{ and  } \sum_{k=1}^p e^{-i 2\t_k}=0.
 \label{unityroots}
\eeq
Averaging over the equations (\ref{mkconjugate} ), from (\ref{unityroots}) we  get
\beq
\label{meannorms}
\begin{array}{l}
\frac{1}{p} \sum_{k=1}^p |m+e^{-i \theta_k}|^2 =|m|^2+1\\
\\
\frac{1}{p} \sum_{k=1}^p e^{-i \theta_k}|m+e^{-i \theta_k}|^2 =m\\
\\
\frac{1}{p} \sum_{k=1}^p e^ {i \theta_k}|m+e^{-i \theta_k}|^2 =\overline{m} .
\end{array}
\eeq
Rewrite the equations of (\ref{alphabeta2}) as
\bdima
\beta= (m+e^{-i \theta_k}) \alpha - m \; |m+e^{-i \theta_k}|^2 -e^{-i \theta_k}|m+e^{-i \theta_k}|^2
\edima
and average over $k=1..p.$ We get
\beq
\label{betaave}
\beta= (m+0) \alpha - m (|m|^2+1) -m= m(\alpha-|m|^2-2).
\eeq
Substituting $\beta =m\alpha-m(|m|^2+2)$ in the left side of (\ref{alphabeta2}) yields
\bdima
m\alpha -m(|m|^2+2) = m \alpha +  e^{-i\t_k} \alpha -m\; |m+e^{-i\t_k}|^2 -e^{-i\t_k}|m+e^{-i\t_k}|^2 \; \mbox{ and }
 \edima
 \bdima
  e^{-i\t_k} \alpha=- m(|m|^2+2)+ m\; |m+e^{-i\t_k}|^2 +e^{-i\t_k}|m+e^{-i\t_k}|^2.
 \edima
 Solve for $\alpha $:
  \bdima
 \alpha=- m(|m|^2+2) e^{i\t_k} + m\; |m+e^{-i\t_k}|^2  e^{i\t_k} + |m+e^{-i\t_k}|^2
 \edima
  Averaging over $k=1..p$, based on (\ref{unityroots}) we obtain
 \bdima
\alpha= - m(|m|^2+2)  0 + m \overline{m} + |m|^2+1= 2|m|^2+1.
\edima
Substituting $ \alpha= 2|m|^2+1$ into  (\ref{betaave}) gives $\beta= m(|m|^2-1),$
concluding the proof of the lemma.
\end{proof}

Rewriting (\ref{alphabeta})  using Lemma \ref{AppLemma}, we get that  $m_k$ are solutions for the equations
\beq
\label{mkcubic}
m(|m|^2-1)= m_k\left (2|m|^2+1 -|m_k(t)|^2 \right), \; \mbox{ for } k=1, \dots p
\eeq
 implying that if the time is fixed, the equations
 $\displaystyle z(2|m|^2+1-|z|^2 )= m(|m|^2-1)$ have at least $p$ distinct complex solutions $z$
 (since $m_k=e^{-i \theta_k}+ m(t)$ are all distinct solutions).

\vskip2mm
\noindent
Case 1: There exists a time  when $  \beta= \displaystyle m(|m|^2-1)$ is nonzero.
We claim that in this case $p\leq 2.$

At a time when $\beta \neq 0$, necessarily
 $\; (2|m|^2+1-|m_k|^2) \neq 0$ . Express $\beta$  using polar coordinates as
$\displaystyle \beta=be^{i\psi}$ where $b>0.$ Since the solutions $z$ to
$z(2|m|^2+1-|z|^2)=be^{i\psi}$ satisfy $ze^{-i\psi} =\frac{b}{2|m|^2+1-|z|^2}\in \R$ , they are collinear (on the line through the origin, making an angle $\psi $ or $\psi +\pi$ with the $x-$ axis).
We conclude that all the points $m+e^{-i\t_k}$ are collinear, which is impossible for the $p^{th}$ roots of unity with $p\geq 3.$

\vskip2mm
\noindent
Case 2:   $\beta=m(|m|^2-1) $ is the zero function on the interval $I_1.$ We claim that this implies that $m(t)=0$ for all $t \in  I_1$  or $p\leq 2.$ 

Assume not: then $p \geq 3 $ and there exists a subinterval  $I_2$ of $I_1$ with $m(t)\neq 0, \beta (t) =0$ for $t$ in $I_2.$ Since $\beta = m(|m|^2-1),$
we get that $|m(t)|=1, \alpha (t)=2|m|^2+1=3 $  and therefore the equations (\ref{mkcubic}) become: at times $t\in I_2$
\beq
\label{betazeromnot3}
0=(m+e^{-i \theta_k}) (3-|m+e^{-i \theta_k}|^2)
\eeq
For a fixed $t,$ the equation $3-|m+e^{-i \theta_k}|^2=0$ can be satisfied by at most two roots of unity, because otherwise we would have three or more roots satisfy $|\overline{m}+e^{i \theta_k}| =\sqrt 3,$
implying that the circle circumscribing those three roots of unity  is centered at $-\overline{m}, $ and has radius $\sqrt 3$ rather than being the unit circle, a contradiction.
Satisfying (\ref{betazeromnot3}) when $p\geq 3$ requires that the factor $(m+e^{-i \theta_k}) $ has a zero of its own, implying that the continuous function $m(t)$ takes values in the discrete set $\{ -e^{-i \theta_k}, k=1..p\}.$  Thus there exists an index, denoted $q,$ such that $m(t)= -e^{-i \theta_q} $ for all $t \in I_2.$ We will show this together with (\ref{defineAppsigma}) imply that  $V= 0, $ contradicting the starting assumption of $V\neq 0.$

Based on $m=-e^{-i \theta_q} $ and (\ref{Apptk}) we get that on $I_2$
\bdima
\dot r_q=  V+e^{i \theta_q}mV=V+e^{i \theta_q}(-1)e^{-i \theta_q} V=0
\edima
meaning that the location  of the $q^{th}$ agent remains unchanged in time.  The equation of motion states
 $\displaystyle \ddot r_q =(1-|\dot r_q|^2) \dot r_q -(r_q-R), $ thus $R=r_q$ is constant, so $\dot R= V=0,$ contradicting $V\neq 0 $ on $ I_2.$

Our two cases show that if $p\geq 3$ then necessarily $m(t)=0$ on the interval $I_1.$ We get that $\dot{\sigma} (t)=0$
 and $r_k (t)=R(t)+ e^{i \theta_k} \sigma _0$ for all $t$ in $I_1.$ Substituting $\dot r _k=V$ and $\ddot r_k= \dot V$ in (\ref{Main})
 yields that for multiple angles $\theta_k$ the equalities $\dot V= (1-|V|^2)V - e^{i \theta_k} \sigma_0$ hold, which require that $\sigma_0=\sigma (t)=0,$ and therefore all $p$ particles' positions coincide on $I_1.$

\end{proof}


\begin{thebibliography}{10}

\bibitem{Turing}
A.~M. Turing.
\newblock The chemical basis of morphogensis.
\newblock {\em Phil. Trans. Roy. Soc. B}, vol. 237, 37--72, 1952.

\bibitem{Kuramoto}
Yoshiki Kuramoto.
\newblock Self-entrainment of a population of coupled non-linear oscillators.
 \newblock {\em Lecture Notes in Physics, International Symposium on Mathematical Problems in Theoretical Physics.} H. Araki (ed.) Vol. 39. Springer-Verlag, New York. p. 420. (1975).


\bibitem{CuckerSmale}
Felipe Cucker and Steve Smale.
\newblock Emergent behavior in flocks
\newblock {\em IEEE Trans. Automat. Control}, vol 52, no. 5, 852--862, 2007.


\bibitem{M-Popovici}
C.~Medynets and I.~Popovici.
\newblock On spatial cohesiveness of second-order self-propelled swarming
  systems.
\newblock {SIAM Journal on Applied Mathematics}, Vol. 83, Iss. 6 2169--2188, 2023.


\bibitem{K-Popovici}
Carl Kolon, Constantine Medynets and Irina Popovici.
\newblock On the stability of Rotating States in Second-Order Self-Propelled Multi-Particle Systems
\newblock {\em preprint arXiv:2105.11419}, 2021.


\bibitem{Topaz}
Chad Topaz and Andrea L. Bertozzi,
 \newblock Swarming Patterns in a Two-Dimensional Kinematic Model for Biological Groups , \newblock {\em SIAM Journal on Applied Mathematics} Volume 65, Number 1, 152--174, 2004.

 \bibitem{BertozziCaltech}
 B. Q. Nguyen, Y-L Chuang, D. Tung, C. Hsieh, Z. Jin, L. Shi, D. Marthaler, A. L. Bertozzi, R. M. Murray.
 \newblock Virtual attractive-repulsive potentials for cooperative control of second order dynamic vehicles on the Caltech MVWT,
 \newblock {\em Proceedings of the American Control Conference}, Portland 2005,  pp. 1084-1089,
 https://www.math.ucla.edu/~bertozzi/papers/potential.pdf.


\bibitem{Bertozzi2006}
Maria Dorsogna, Yao-Li Chuang, Andrea Bertozzi, and L~Chayes.
\newblock Self-propelled particles with soft-core interactions: Patterns,
  stability, and collapse.
\newblock {\em Physical review letters}, 96:104302, 04 2006.




\bibitem{vonBrecht}
G.~Albi, D.~Balagué, J.~A. Carrillo, and J.~von Brecht.
\newblock Stability analysis of flock and mill rings for second order models in
  swarming.
\newblock {\em SIAM Journal on Applied Mathematics}, 74(3):794--818, 2014.

\bibitem{Liu}
Seung-Yeal Ha and Jian-Guo Liu.
\newblock A simple proof of the Cucker-Smale flocking dynamics and mean-field limit.
\newblock {\em Communications in Mathematical Sciences}, vol. 7, no. 2, 297 -- 325, 2009.


\bibitem{Carr}
Jack Carr.
\newblock {\em Applications of centre manifold theory}, volume~35 of {\em
  Applied Mathematical Sciences}.
\newblock Springer-Verlag, New York-Berlin, 1981.

\bibitem{Vander}
Vanderbauwhede, A.
\newblock{\em Centre Manifolds, Normal Forms and Elementary Bifurcations},  In: Kirchgraber, U., Walther, H.O. (eds) Dynamics Reported. Dynamics Reported, vol 2.
\newblock Vieweg+Teubner Verlag, Wiesbaden, 1989.

\bibitem{C-S-Martin}
{J A} Carrillo, Y~Huang, and S~Martin.
\newblock Nonlinear stability of flock solutions in second-order swarming
  models.
\newblock {\em Nonlinear Analysis: Real World Applications}, 17, 332--343, 2014.




\bibitem{Haraux}
Alain Haraux and Mohamed~Ali Jendoubi.
\newblock {\em The convergence problem for dissipative autonomous systems}.
\newblock SpringerBriefs in Mathematics. Springer, Cham; BCAM Basque Center for
  Applied Mathematics, Bilbao, 2015.
\newblock Classical methods and recent advances, BCAM SpringerBriefs.


\bibitem{Bertozzi2011}
Theodore Kolokolnikov, Hui Sun, David Uminsky, and Andrea~L. Bertozzi.
\newblock Stability of ring patterns arising from two-dimensional particle
  interactions.
\newblock {\em Phys. Rev. E}, 84:015203, Jul 2011.

\bibitem{Bertozzi2007}
Yao li~Chuang, Maria~R. D’Orsogna, Daniel Marthaler, Andrea~L. Bertozzi, and
  Lincoln~S. Chayes.
\newblock State transitions and the continuum limit for a 2d interacting,
  self-propelled particle system.
\newblock {\em Physica D: Nonlinear Phenomena}, 232(1):33--47, 2007.

\bibitem{G-Holmes}
John Guckenheimer and Philip Holmes.
\newblock {\em Nonlinear Oscillations, Dynamical Systems, and Bifurcations of Vector Fields}.
\newblock Springer New York, NY, 1983.

\bibitem{Bertozzi2015}
A.L. Bertozzi, T. Kolokolnikov, H. Sun, D. Uminsky and J. von Brecht, J.
\newblock Ring patterns and their bifurcations in a nonlocal model of biological swarms. \newblock {\em Communications in Mathematical Sciences}, Volume 13 (4). Pages: 955-985 , 2015.

\bibitem{Carrillo}
J. Carrillo, M. D’orsogna, and V. Panferov,
\newblock Double milling in self-propelled swarms from kinetic theory,
\newblock {\em Kinetic and Related Models}, 2 , 363–-378, 2009.


\bibitem{K-H-Strogatz}
Kevin~P. O'Keeffe, Hyunsuk Hong, and Steven~H. Strogatz.
\newblock Oscillators that sync and swarm.
\newblock {\em Nature Communications}, 8(1):1504, 2017.

\bibitem{Pikovsky}
Arkady Pikovsky, Michael Rosenblum, and J{\"u}rgen Kurths.
\newblock {\em Synchronization: a universal concept in nonlinear sciences},
  volume~12 of {\em Camb. Nonlinear Sci. Ser.}
\newblock Cambridge: Cambridge University Press, reprint of the 2001 hardback
  edition edition, 2003.

\bibitem{Schwartz}
Klementyna Szwaykowska, Ira~B. Schwartz, Luis~Mier y~Teran~Romero,
  Christoffer~R. Heckman, Dan Mox, and M.~Ani Hsieh.
\newblock Collective motion patterns of swarms with delay coupling: Theory and
  experiment.
\newblock {\em Physical Review E}, 93(3), mar 2016.


\bibitem{Ballerini}
Ballerini M, Cabibbo N, Candelier R, Cavagna A, Cisbani E, Giardina I, Lecomte V, Orlandi A, Parisi G, Procaccini A, Viale M, Zdravkovic V.
\newblock Interaction ruling animal collective behavior depends on topological rather than metric distance: evidence from a field study.
\newblock {\em Proc Natl Acad Sci U S A.}, vol 105(4):1232--7,  2008.

\bibitem{Camazine}
Scott Camazine, Jean-Louis Deneubourg, Nigel R. Franks, James Sneyd, Guy Theraula, and Eric Bonabeau
\newblock Self-Organization in Biological Systems
\newblock {\em Princeton Studies in Complexity}, 2003.

\bibitem{Buhl}
J. Buhl, David J. T. Sumpter, Iain D. Couzin, Joseph J. Hale, Emma Despland, Esther R. Miller and Stephen J. Simpson,
 \newblock From Disorder to Order in Marching Locusts
 \newblock {\em Science}, vol. 312, 1402 --1406, 2006.

\bibitem{Reppert}
C. Liu, D.R. Weaver, S.H. Strogatz, and S.M. Reppert.
\newblock Cellular Construction of a Circadian Clock: Period Determination in the Suprachiasmatic Nuclei
\newblock {\em Cell} 91, 855--860, 1997.


\bibitem{Hellmann}
Hellmann, F., Schultz, P., Jaros, P. et al.
\newblock Network-induced multistability through lossy coupling and exotic solitary states.
\newblock {\em  Nat Commun} 11, 592, 2020.


\bibitem{Wang}
Herbert Winful and S. S. Wang.
\newblock Stability of phase locking in coupled semiconductor laser arrays.
 \newblock {\em Applied Physics Letters}, 53(20): 1894--1896, 1988.


\bibitem{Bridge}
S. H. Strogatz, D. M. Abrams, A. McRobie, B. Eckhardt, E. Ott.
\newblock Theoretical Mechanics: Crowd synchrony on the millennium bridge.
\newblock {\em Nature} 438 (3), pp. 43–44, 2005.


\bibitem{Kim}
Seung-Yeal Ha, Taeyoung Ha, and  Jong-Ho Kim,
\newblock On the complete synchronization of the Kuramoto phase model,
\newblock {\em Physica D: Nonlinear Phenomena}, Volume 239, Issue 17, 1692--1700, 2010.

\bibitem{Chopra}
N. Chopra and M. W. Spong,
\newblock On Synchronization of Kuramoto Oscillators,
\newblock {\em Proceedings of the 44th IEEE Conference on Decision and Control}, Seville, Spain, 3916--3922,  2005.

\bibitem{Carrillo2}
J.A. Carrillo, D. Kalise, F. Rossi, E. Trelat,
\newblock Controlling Swarms toward Flocks and Mills,
\newblock {\em SIAM Journal on Control and Optimization} , Vol. 60, Iss.3,  1863--1891, 2022.



\bibitem{Zuyev}
Grushkovskaya, V.,  Zuyev, A.
\newblock Asymptotic behavior of solutions of a nonlinear system in the critical case of q pairs of purely imaginary eigenvalues,
\newblock {\em Nonlinear Analysis} 80(2013), 156–- 178.



\bibitem{Krasilnikov}
P.S. Krasilnikov,
\newblock Algebraic criteria for asymptotic stability at 1- 1 resonance,
\newblock {\em Journal of Applied Mathematics and Mechanics}, Volume 57, Issue 4, 1993, 583--590.


\bibitem{Grushkovskaya}
V. Grushkovskaya,
\newblock On the influence of resonances on the asymptotic behavior of trajectories of nonlinear systems in critical cases,
\newblock {\em Nonlinear Dynamics}, 86 (2016), 587--603.


\bibitem{Corduneanu}
C. Corduneanu with N. Gheorghiu and V. Barbu,
\newblock Almost periodic functions,  Chapter I,
\newblock {\em Tracts in Mathematics}, Translated from the Romanian ed. by G. Bernstein and E. Tomer, 1961.

\end{thebibliography}
\end{document}